\documentclass[a4paper,reqno,11pt]{amsart}

\usepackage{graphicx} 
\usepackage{amssymb}
\usepackage{amstext}
\usepackage{amsmath}
\usepackage{amscd}
\usepackage{amsthm}
\usepackage{amsfonts}
\usepackage{enumerate}
\usepackage{color}
\usepackage{latexsym}
\usepackage{pstricks}
\usepackage[all,2cell,color]{xy}
\usepackage{comment}
\usepackage{todonotes}
\usepackage{tikz}
\usetikzlibrary{matrix,arrows,calc,snakes,patterns,decorations.markings}
\usepackage{booktabs}
\usepackage{multirow}

\numberwithin{equation}{section}
\newtheorem{theorem}{Theorem}[section]

\newtheorem{lemma}[theorem]{Lemma}
\newtheorem{proposition}[theorem]{Proposition}
\newtheorem{definition-theorem}[theorem]{Definition-Theorem}
\newtheorem{definition-proposition}[theorem]{Definition-Proposition}
\newtheorem{problem}[theorem]{Problem}

\theoremstyle{definition}
\newtheorem{definition}[theorem]{Definition}

\newtheorem{example}[theorem]{Example}

\newtheorem{convention}[theorem]{Convention}

\newcommand{\HH}{\mathsf{H}}
\newcommand{\m}{\mathfrak{m}}

\newcommand{\p}{\mathrm{p}}
\newcommand{\n}{\mathrm{n}}

\newcommand{\Hom}{\operatorname{Hom}\nolimits}

\newcommand{\End}{\operatorname{End}\nolimits}

\newcommand{\RHom}{\mathbf{R}\strut\kern-.2em\operatorname{Hom}\nolimits}

\DeclareMathOperator{\moduleCategory}{\mathsf{mod}} \renewcommand{\mod}{\moduleCategory}

\DeclareMathOperator{\proj}{\mathsf{proj}}

\DeclareMathOperator{\thick}{\mathsf{thick}}

\DeclareMathOperator{\add}{\mathsf{add}}

\newcommand{\Db}{\mathsf{D}^{\rm b}}
\newcommand{\Kb}{\mathsf{K}^{\rm b}}

\newcommand{\val}{{\rm val}}
\newcommand{\occ}{{\rm occ}}

\begin{document}

\title{Mutation of Brauer configuration algebras}

\author[Toshitaka Aoki]{Toshitaka Aoki}
\address{Toshitaka Aoki,
Graduate School of Human Development and Environment, Kobe University, 3-11 Tsurukabuto, Nada-ku, Kobe 657-8501, Japan}
\email{toshitaka.aoki@people.kobe-u.ac.jp}

\author[Yingying Zhang]{Yingying Zhang*}
\address{Yingying Zhang, 
Department of Mathematics, Huzhou University, Huzhou 313000, Zhejiang Province, P.R.China}
\email{yyzhang@zjhu.edu.cn}
\thanks{*Corresponding author}

\begin{abstract}
For Brauer graph algebras, tilting mutation is compatible with flip of Brauer graphs. 
The aim of this paper is to generalize this result to the class of Brauer configuration algebras introduced by Green and Schroll recently. 
More precisely, under a certain condition, we introduce flip of Brauer configurations and prove that it is compatible with tilting mutation of the corresponding Brauer configuration algebras. 
\end{abstract}

\maketitle

\section{Introduction}
Tilting theory is a powerful tool to analyze the structure of derived categories and a rich source of derived equivalences. A central notion is played by tilting objects, which can be seen as a generalization of progenerators in Morita theory \cite{Rickard89Morita}. 
The procedure called tilting mutation produces a new tilting object from the original one by replacing a chosen summand if possible, and it brings a combinatorial framework in the study \cite{BGP73,RS91,Okuyama98,HU05}. 
To complement this operation, silting mutation on the class of silting complexes was introduced by \cite{AI12}, where any summand of a silting complex can be always replaced to obtain a new silting complex. 
Silting mutation has close relation to various important notions in representation theory \cite{FZ02,IY08,BY13,AIR14,Zhang17,Cao21}. 
In the case that a given algebra is symmetric, 
all silting complexes are tilting, and tilting mutation yields large family of derived equivalent symmetric algebras as the endomorphism algebras. 
Thus, it is a fundamental problem to describe tilting mutation of a given symmetric algebra.

Brauer graph algebras, originating in modular representation theory of finite groups, form one of the most important classes of finite dimensional symmetric algebras. 
A Brauer graph algebra is defined from a combinatorial object called a Brauer graph. A Brauer graph consists of a ribbon graph -- a finite graph equipped with a cyclic ordering of the edges incident to each vertex -- and the multiplicity on each vertex. 
From its definition, there is a realization on oriented Riemann surfaces \cite{Labourie13,MS14} and a combinatorial/geometric approach based on ribbon graphs is widely used. 
Especially, Aihara \cite{Aihara15} introduced an operation of flip (with respect to a chosen edge), providing a construction of a new Brauer graph from the original one and showed the following result.

\begin{theorem}[\cite{Aihara15}, see also \cite{Kauer98}]
\label{intro:flip of BGA}
    Tilting mutation of Brauer graph algebras is compatible with flip of Brauer graphs. 
\end{theorem}

We remark that the class of Brauer graph algebras is closed under derived equivalence \cite{AZ22} (see also \cite{OZ22}). 
For more research on Brauer graph algebras, see \cite{Alperin86,Benson91,Antipov07,Aihara14,AAC18, Schroll18,ES20,WZ22,AY23}.

In \cite{GS16,GS17}, Green and Schroll introduced a new class of symmetric algebras called \emph{Brauer configuration algebras} 
and proved that the class of Brauer configuration algebras coincides with that of symmetric special multiserial algebras. 
A Brauer configuration algebra is defined from a combinatorial object called a \emph{Brauer configuration}, which is a generalization of Brauer graphs. 
Indeed, Brauer graphs are exactly Brauer configurations all of whose polygons are $2$-gons. 
Brauer configuration algebras include the well-studied class of symmetric algebras with radical cube zero \cite{Benson08,HK01,ES11,AA23}. 
Moreover, Brauer graph algebras are known to be tame, while tame Brauer configuration algebras are classified by \cite{Duffield18}. 

It is an important problem to study tilting mutation of a given Brauer configuration algebra $A_{\Gamma}$ in terms of its Brauer configuration $\Gamma$. 
Our starting point is the following observation in contrast to the case of Brauer graph algebras. 
Now, let $P_V$ be the indecomposable projective $A_{\Gamma}$-module corresponding to a polygon $V$ of $\Gamma$. 

\begin{proposition}[\rm Example \ref{example:BCAtilting}]
    Tilting mutation $\End(\mu^-_{P_V}(A_{\Gamma}))$ is not a Brauer configuration algebra in general. 
    Hence, the class of Brauer configuration algebras is not closed under derived equivalence. 
\end{proposition}

Our main result gives a sufficient condition for a tilting mutation of Brauer configuration algebras to be Brauer configuration algebras again. 
More precisely, we are able to introduce the notion of flip of Brauer configurations under the following condition (E).

\begin{definition}[See Definitions \ref{def:conditionE} and \ref{def:flip BC} for the details]
    We say that a polygon $V$ of $\Gamma$ satisfies the condition (E) 
    if every predecessor of $V$ in the cyclic ordering around each vertex is either an edge or $V$ itself. 
    In this case, we define a new Brauer configuration $\mu^-_V(\Gamma)$, which we call a (left) \emph{flip} of $\Gamma$ at $V$. 
\end{definition}

Dually, one can define a right flip. 
A typical example of flip is visualized in the following Figure \ref{fig:intro_flip}, where we describe the cyclic ordering around vertices in a counterclockwise direction. 
\begin{figure}[h]
\begin{tabular}{cccccc}
\begin{tikzpicture}
\node at(0,0) {$
\begin{tikzpicture}[scale = 0.65, baseline =0mm]
    \coordinate (0) at (0*36:1.5);
    \coordinate (1) at (1*36:1.5);
    \coordinate (2) at (2*36:1.5);
    \coordinate (3) at (3*36:1.5);
    \coordinate (4) at (4*36:1.5);
    \coordinate (5) at (5*36:1.5);
    \coordinate (6) at (6*36:1.5);
    \coordinate (7) at (7*36:1.5);
    \coordinate (8) at (8*36:1.5);
    \coordinate (9) at (9*36:1.5);
    
    \coordinate (t1) at ($(0)!0.5!(2)$);
    \draw[thick] (0)--(1) (2)--(3) (4)--(5) (6)--(7) (8)--(9); 
    \draw[line width =0mm,fill = white!20!lightgray] (0)--(2)--(4)--(6)--(8)--cycle;
    \draw[thick] (0)--(2)--(4)--(6)--(8)--cycle;
    \node at (0) {$\bullet$};
    \node at (1) {$\bullet$};
    \node at (2) {$\bullet$};
    \node at (3) {$\bullet$};
    \node at (4) {$\bullet$};
    \node at (5) {$\bullet$};
    \node at (6) {$\bullet$};
    \node at (7) {$\bullet$};
    \node at (8) {$\bullet$};
    \node at (9) {$\bullet$};
    
    \node at (0,0) {$V$};
\end{tikzpicture}$};
\node at(5,0) {$
\begin{tikzpicture}[scale = 0.65, baseline = 0mm]
    \coordinate (0) at (0*36:1.5);
    \coordinate (1) at (1*36:1.5);
    \coordinate (2) at (2*36:1.5);
    \coordinate (3) at (3*36:1.5);
    \coordinate (4) at (4*36:1.5);
    \coordinate (5) at (5*36:1.5);
    \coordinate (6) at (6*36:1.5);
    \coordinate (7) at (7*36:1.5);
    \coordinate (8) at (8*36:1.5);
    \coordinate (9) at (9*36:1.5);
    
    \coordinate (t1) at ($(0)!0.5!(2)$);
    \draw[thick] (0)--(1) (2)--(3) (4)--(5) (6)--(7) (8)--(9); 
    \draw[line width =0mm, fill = white!20!lightgray] (1)--(3)--(5)--(7)--(9)--cycle;
    \draw[thick] (1)--(3)--(5)--(7)--(9)--cycle;
    \node at (0) {$\bullet$};
    \node at (1) {$\bullet$};
    \node at (2) {$\bullet$};
    \node at (3) {$\bullet$};
    \node at (4) {$\bullet$};
    \node at (5) {$\bullet$};
    \node at (6) {$\bullet$};
    \node at (7) {$\bullet$};
    \node at (8) {$\bullet$};
    \node at (9) {$\bullet$};
    
    \node at (0,0) {$V$};
\end{tikzpicture}$};
\draw[->, fill=white,line width = 1.5pt] (1.7,0)--node[above]{\text{\small left flip at $V$}}(3.3,0); 
\draw[<-, fill=white,line width = 1.5pt] (1.7,-0.3)--node[below]{\text{\small right flip at $V$}}(3.3,-0.3); 

\end{tikzpicture}
&  
\begin{tikzpicture}[baseline = -12mm]
\node at(0,0) {$
\begin{tikzpicture}[scale = 0.65, baseline = 0mm]
        \coordinate (a) at (-1,0); 
        \coordinate (b) at (0.5,0);
        \coordinate (c) at (-1,2.4);
        
\draw[line width =0mm, thick, fill = white!20!lightgray] (a)..controls($(a)+(15:1.9)$)and($(c)+(0:1)$)..(c)..controls($(c)+(180:1)$)and($(a)+(165:1.9)$)..(a);
\draw[line width =0mm, thick] (a)..controls($(a)+(15:1.9)$)and($(c)+(0:1)$)..(c)..controls($(c)+(180:1)$)and($(a)+(165:1.9)$)..(a);

\draw[fill=white, line width =0mm, thick] (a)..controls($(a)+(35:1.7)$)and($(a)+(95:1.7)$)..(a); 
\draw[fill=white, line width =0mm, thick] (a)..controls($(a)+(85:1.7)$)and($(a)+(145:1.7)$)..(a);
\draw[thick] (a)--(b); 
\node at (a) {$\bullet$};
\node at (b) {$\bullet$};
\node at (c) {$\bullet$}; 
\node at (-1, 1.5) {$U$};
\end{tikzpicture}$}; 
\node at(4.7,0) {$
\begin{tikzpicture}[scale = 0.65, baseline = 0mm]
        \coordinate (a) at (-1,0); 
        \coordinate (b) at (0.5,0);
        \coordinate (c) at (0.5,2.4);
        
\draw[line width =0mm, thick, fill = white!20!lightgray] (b)..controls($(b)+(15:1.9)$)and($(c)+(0:1)$)..(c)..controls($(c)+(180:1)$)and($(b)+(165:1.9)$)..(b);
\draw[line width =0mm, thick] (b)..controls($(b)+(15:1.9)$)and($(c)+(0:1)$)..(c)..controls($(c)+(180:1)$)and($(b)+(165:1.9)$)..(b);

\draw[fill=white, line width =0mm, thick] (b)..controls($(b)+(35:1.7)$)and($(b)+(95:1.7)$)..(b); 
\draw[fill=white, line width =0mm, thick] (b)..controls($(b)+(85:1.7)$)and($(b)+(145:1.7)$)..(b);
\draw[thick] (a)--(b); 
\node at (a) {$\bullet$};
\node at (b) {$\bullet$};
\node at (c) {$\bullet$}; 
\node at (0.5, 1.5) {$U$};
\end{tikzpicture}$}; 
\draw[->, fill=white, line width = 1.5pt] (1.6,0)--node[above]{\text{\small left flip at $U$}}(3.2,0);
\draw[<-, fill=white, line width = 1.5pt] (1.6,-0.3)--node[below]{\text{\small right flip at $U$}}(3.2,-0.3);
\end{tikzpicture}
\end{tabular}
    \caption{Flip at a $5$-gon $V$ in the left and a (self-folded) $4$-gon $U$ in the right.}
    \label{fig:intro_flip}
\end{figure}
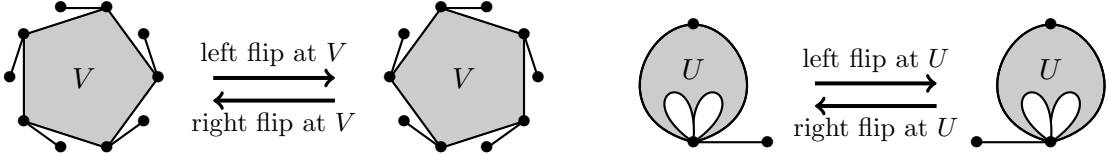

Now, we state our main result. 

\begin{theorem}[Theorem \ref{thm:main theorem}] 
\label{main theorem}
Let $\Gamma$ be a Brauer configuration and $A_{\Gamma}$ the Brauer configuration algebra of $\Gamma$ over an algebraically closed field $K$. 
For a polygon $V$ of $\Gamma$, if it satisfies the condition {\rm (E)}, we have 
\begin{equation*}
    \End(\mu_{P_{V}}^{-}(A_{\Gamma}))\cong A_{\mu_V^-(\Gamma)}.
\end{equation*}
In particular, it is again a Brauer configuration algebra. 
\end{theorem}

Theorem \ref{intro:flip of BGA} is a special case of our theorem. Indeed, when $\Gamma$ is a Brauer graph, the condition (E) is automatically satisfied for all edges, and our flip coincides with that of Brauer graphs. 

We pose the next problem. 

\begin{problem}
    Is there a reasonable class of algebras which contains Brauer configuration algebras and is closed under derived equivalence?  
\end{problem}

This paper is organized as follows. 
In Section \ref{sec:preliminaries}, we recall basic definition of tilting theory and Brauer configuration algebras. 
In Section \ref{sec:main result}, we define flip of Brauer configurations (Definition \ref{def:flip BC}) under the condition (E) and prove our main result (Theorem \ref{thm:main theorem}).

\bigskip
{\bf Notation. }
Let $K$ be an algebraically closed field. We denote by $D:=\Hom_K(-,K)$ the $K$-duality. For a finite dimensional $K$-algebra $A$, we denote by $\mod A$ (resp. $\proj A$) the category of finitely generated (resp. finitely generated projective) right $A$-modules. We denote by $\Kb(\proj A)$ the homotopy category of bounded complexes of finitely generated projective $A$-modules. We denote by $\Db(A)$ the bounded derived category of $\mod A$. Throughout, all morphisms are taken in the derived category unless otherwise specified. Finally, we denote by $\#\mathcal{S}$ the cardinality of a given set $\mathcal{S}$.

\section{Preliminaries}\label{sec:preliminaries}

We refer to \cite{ASS06, ARS95} for basic terminology of representation theory of finite dimensional algebras and to \cite{Happel88} for definitions of triangulated categories. 

\subsection{Tilting theory}
Let $A$ be a finite dimensional $K$-algebra.

\begin{definition}
Let $T=(T^{i},d^i)$ be a bounded complex of $\Kb(\proj A)$. 
\begin{enumerate}
    \item We say that $T$ is \emph{presilting} (resp., \emph{pretilting}) if $\Hom(T,T[i]) =0$ for all integer $i>0$ (resp., $i\neq 0$). 
    \item We say that $T$ is \emph{silting} (resp., \emph{tilting}) if it is presilting (resp., pretilting) and 
    $\thick T = \Kb(\proj A)$, where $\thick T$ is the smallest triangulated full subcategory of $\Kb(\proj A)$ which is closed under taking direct summands. 
    \item We say that $T$ is \emph{two-term} if $T^i=0$ for all $i\neq 0,-1$. 
\end{enumerate}
\end{definition}

The following result is basic in tilting theory. 

\begin{theorem}\label{derived-equivalent}{\rm \cite[Theorem 6.4]{Rickard89Morita}}
    Let $A$ and $B$ be finite dimensional $K$-algebras. 
    Then, the following conditions are equivalent. 
    \begin{enumerate}[\rm (a)]
        \item $\Db(A) \cong \Db(B)$ (In this case, we say that $A$ is derived equivalent to $B$). 
        \item There exists a tilting complex $T \in \Kb(\proj A)$ such that $B\cong \End(T)$.
    \end{enumerate}
\end{theorem}

Next, we recall the notion of silting mutation. 

\begin{definition-proposition}\rm \cite[Theorem 2.31]{AI12}
Let $T=X\oplus Q$ be a basic silting complex in $\Kb(\proj A)$ with an indecomposable direct summand $X$. 
We take a triangle
\[
X\buildrel {f} \over\longrightarrow Q'\longrightarrow Y\longrightarrow X[1],  
\] 
where $f$ is a left minimal ($\add Q$)-approximation of $X$. 
Then, $Y$ is indecomposable.
In this case, $\mu_X^-(T) := Y\oplus Q$ is again a basic silting complex and called \emph{left mutation} of $T$ with respect to $X$. 
The {\it right mutation} $\mu^+_X(T)$ is defined dually. 
\end{definition-proposition}

\begin{example}
    Let $A$ be a finite dimensional algebra. 
    For each indecomposable projective $A$-module $P$, 
    we obtain a left mutation $\mu_{P}^{-}(A)$, which is  silting. 
    If it is tilting, then its endomorphism algebra $\End(\mu_{P}^{-}(A))$ is known as \emph{tilting mutation} of the algebra $A$. 
    Notice that it is derived equivalent to $A$ by Theorem \ref{derived-equivalent}.
\end{example}

We use the following results to compute the dimension of $K$-vector spaces. Now, for two complexes 
$T,U\in \Kb(\proj A)$, we set $(T,U) := \dim_K\Hom(T,U)$. 
For the case when $A$ is symmetric, the Nakayama functor $\nu:=D\Hom_A(-,A)$ is the identity. 
Thus, we obtain 
\begin{equation}\label{Serre duality}
    \Hom(P,X) \cong D\Hom(X,P) \quad \text{and} \quad (P,X) = (X,P)
\end{equation}
for two complexes $P,X\in \Kb(\proj A)$. 
This shows that silting complexes coincide with tilting complexes for symmetric algebras (see also \cite[Example 2.8]{AI12}). 
In addition, we can deduce from \cite{Happel88} that, if $T\oplus U\in \Kb(\proj A)$ is a two-term pretilting complex, then we have 
\begin{equation}\label{dim formula}
    (T,U) = (U,T) = 
    (T^0,U^0) - 2(T^{-1},U^{0}) + (T^{-1},U^{-1}). 
\end{equation}

\subsection{Brauer configurations}

We recall the definition of Brauer configurations from \cite{GS16, GS17}, which generalizes that of Brauer graphs (see \cite{Schroll18} for the basic knowledge of Brauer graphs). 


\begin{definition}\label{def:BC}
A \emph{Brauer configuration} consists of 
\begin{itemize}
    \item a non-empty finite set $\Gamma_{0}$, 
    \item a non-empty finite collection $\Gamma_{1}$ of finite labeled multisets whose elements belong to $\Gamma_0$, 
    \item an assignment $\mathfrak{o}$ which assigns each $u\in \Gamma_0$ to a chosen cyclic ordering on elements of $\Gamma_1$ that contain $u$, including repetitions, and 
    \item a function $\m \colon \Gamma_{0} \to \mathbb{Z}_{>0}$.  
\end{itemize}
We further require that every element of $\Gamma_1$ has at least two elements and every element of $\Gamma_0$ appears in at least one element of $\Gamma_1$.
\end{definition}

For our convenience, we introduce the following definition which is equivalent to the definition of Brauer configurations. 
We should remark that the use of this combinatorial language is inspired by the ribbon graph theory for Brauer graph algebras from \cite{AAC18, MS14}.

\begin{definition} \label{def:our BC}
Let $\Gamma := (H, \sigma, \psi, s, \m)$ be a tuple, where $H$ is a non-empty finite set, 
$\sigma$ is a permutation on $H$, and 
$\psi$ is an equivalent relation on $H$ each of whose equivalence classes has at least two elements. 
\begin{enumerate}[\rm (1)]
    \item Each element of $H$ is called an \emph{angle} of $\Gamma$. 
    \item Each equivalence class in $H/\psi$ is called a \emph{polygon} of $\Gamma$. 
    For $m\geq 2$, it is called \emph{$m$-gon} if $m$ is its cardinality. 
    We denote by $[h]$ the polygon containing the angle $h\in H$.
    \item Let $s \colon H \to H/\langle \sigma \rangle$ be a canonical surjection. 
    Each element of $H/\langle \sigma \rangle$ is called a \emph{vertex} of $\Gamma$. 
    For a vertex $u$ of $\Gamma$, the \emph{valency} $\val(u)$ is defined by $\val(u) := \#s^{-1}(u)$. 
    \item For each vertex $u$ of $\Gamma$, 
    the $\sigma$-orbit $$(h,\sigma(h),\ldots,\sigma^{\val(u)-1}(h))$$ incident to $u$ is called the \emph{cyclic ordering} around $u$. 
    \item $\m \colon H/\langle\sigma\rangle \to \mathbb{Z}_{>0}$ is a function which we call the \emph{multiplicity function}. 
\end{enumerate}
\end{definition}

Let $\Gamma:=(H,\sigma,\psi,s,\m)$ be the above tuple. 
Then, it gives rise to a Brauer configuration in the sense of Definition \ref{def:BC} as follows: 
\begin{itemize}
    \item Let $\Gamma_0$ be the set $H/\langle \sigma \rangle$ of vertices of $\Gamma$.
    \item Let $\Gamma_1$ be a collection of 
    labeled multisets of the form $\{s(h) \mid h\in V\}$ for all polygons $V\in H/\psi$ of $\Gamma$. 
    \item For each $v\in \Gamma_0$, the cyclic ordering around $v$ is given by the $\sigma$-orbit incident to $v$. 
    \item A function $\m$ is the multiplicity function of $\Gamma$. 
\end{itemize}
Conversely, such a tuple can be obtained from any Brauer configuration. Thus, we impose the following convention.

\begin{convention}
Throughout this paper, by a Bauer configuration, we mean a tuple $(H, \sigma, \psi,s,\m)$ given in Definition \ref{def:our BC}.  
In addition, by a vertex, a polygon, etc. of a Brauer configuration, we mean those in Definition \ref{def:our BC}.
\end{convention}

We use the following notations.

\begin{definition}\label{def:basicBC}
    Let $\Gamma=(H,\sigma,\psi,s,\m)$ be a Brauer configuration. 
    \begin{enumerate}[\rm (1)]
        \item For a vertex $u$ and a polygon $V$ of $\Gamma$, 
        let $\occ(u,V) := \#\{h\in V \mid s(h)=u\}$. 
        By definition,  
        $\val(u)$ coincides with the sum of $\occ(u,V)$ for all polygons $V$. 
        \item We say that a polygon $V$ is \emph{self-folded} if there exists at least one vertex $u$ such that $\occ(u,V)>1$. 
        \item A vertex $u$ is called \emph{external} if $\val(u)=1$, and \emph{truncated} if $\val(u)=\m(u)=1$. 
        \item A polygon $V$ is called \emph{edge} if it is a $2$-gon, i.e., $V=\{h,h'\}$ for some $h,h'\in H$. 
        In this case, let $\bar{h} := h'$ and $\bar{h'} := h$. Thus, $h = \bar{\bar{h}}$ holds. 
        \item We say that $\Gamma$ is \emph{multiplicity-free} if $\mathfrak{m}(u) = 1$ for all vertices $u$. 
    \end{enumerate}
\end{definition}

In order to describe a given Brauer configuration $\Gamma=(H,\sigma,\psi,s,\m)$, we use the following convention. 
An $m$-gon $V=\{h_1,\ldots,h_m\}$ of $\Gamma$ is realized as an actual $m$-gon whose internal angles are labeled by $h_1,\ldots,h_m$. Notice that there are several choices of such labeling. 
When we describe the cyclic ordering $(h,\sigma(h),\ldots,\sigma^{\val(u)-1}(h))$ of angles around the vertex $u=s(h)$, we draw them in the plane locally so that the angles
$h,\sigma(h)\ldots,\sigma^{\val(u)-1}(h)$ 
appear around $u$ in this order counterclockwisely.

We give an example of Brauer configurations.

\begin{example}\label{example:BC}
We consider a Brauer configuration defined by the following data. 
Let $H := \{a_1,a_2,a_3,a_4,a_5,a_6,a_7,b,\bar{b},c,\bar{c},d,\bar{d},e,\bar{e},f,\bar{f},g,\bar{g}\}$ be the set of angles. 
Let $\psi$ be an equivalent relation on $H$ 
such that the set of polygons is given by $\{U_1\ldots, U_7\}$, 
where 
$U_1 := \{a_1,\ldots, a_7\}$,  
$U_2 := \{b,\bar{b}\}$,  
$U_3 := \{c,\bar{c}\}$,  
$U_4 := \{d,\bar{d}\}$,  
$U_5 := \{e,\bar{e}\}$,  
$U_6 := \{f,\bar{f}\}$ and    
$U_7 := \{g,\bar{g}\}$. 
In addition, let $\sigma$ be a permutation on $H$ such that 
the cyclic orderings around vertices are given by 
\begin{gather*}
    w_1 = (a_1,c,\bar{e},\bar{d},e,d), \ w_2 = (a_2,b,\bar{c}), \ w_3 = (a_3,a_4,a_5,\bar{b}), \\ 
    w_4 = (a_6), \ w_5 = (a_7,f,g), \ w_6 = (\bar{f}), \ w_7 = (\bar{g}).  
\end{gather*}
In particular, the polygons $U_1$, $U_4$ and $U_5$ are self-folded. 
Finally, we set $\mathfrak{m}(w_4) = 2$ and $\mathfrak{m}(w_i) = 1$ for all $i\neq 4$. 
Then, they define a Brauer configuration $\Gamma = (H,\sigma,\psi,s,\mathfrak{m})$, where $s$ is a canonical surjection associated with $\sigma$. 
A realization of $\Gamma$ is given by the following figure. 

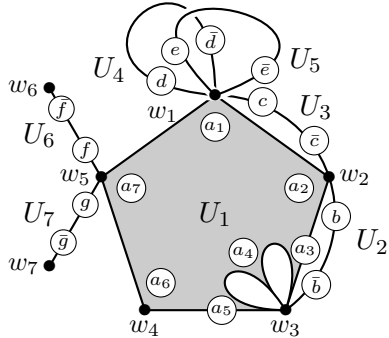
\vspace{-14mm}
\begin{equation*}
\begin{tikzpicture}[scale =1.05]
\coordinate (1) at (90-72:1.5);
\coordinate (2) at (90:1.5); 
\coordinate (3) at (90+72*1:1.5);
\coordinate (4) at (90+72*2:1.5); 
\coordinate (5) at (90+72*3:1.5);
\coordinate (3a) at ($(3)+(120:1.3)$);
\coordinate (3b) at ($(3)+(-120:1.3)$);
\fill[line width =0mm, fill = white!20!lightgray] (1)--(2)--(3)--(4)--(5)--cycle; 

\draw[thick] (2)arc(90:90-72*2:1.5);
\draw[thick] (1)--(2)--(3)--(4)--(5)--cycle; 
\draw[fill=white, line width =0mm, thick] (5)..controls($(5)+(72:1.3)$)and($(5)+(72+108-50:1.3)$)..(5);
\draw[fill=white, line width =0mm, thick] (5)..controls($(5)+(72+50:1.3)$)and($(5)+(72+108:1.3)$)..(5);

\draw[thick] (3)--(3a); 
\draw[thick] (3)--(3b);

\draw[white, line width = 1.5mm] (2)..controls($(2)+(90:2.5)$)and($(2)+(90+85:2.5)$)..(2);
\draw[thick] (2)..controls($(2)+(90:2.5)$)and($(2)+(90+85:2.5)$)..(2);
\draw[white, line width = 1.5mm] (2)..controls($(2)+(90-70:2.7)$)and($(2)+(90+45:2.7)$)..(2);
\draw[thick] (2)..controls($(2)+(90-70:2.7)$)and($(2)+(90+45:2.7)$)..(2);

\node at(1) {$\bullet$}; 
\node at(2) {$\bullet$}; 
\node at(3) {$\bullet$}; 
\node at(4) {$\bullet$}; 
\node at(5) {$\bullet$}; 
\node at(3a) {$\bullet$};
\node at(3b) {$\bullet$};
\node at (0,0) {$U_1$}; 
\node at (-10:2) {$U_2$}; 
\node at (47:1.85) {$U_3$}; 
\node at ($(124:2.2)+(180:0.1)$) {$U_4$}; 
\node at ($(62:2.2)+(0:0.1)$) {$U_5$}; 
\node at (-2.2,1) {$U_6$}; 
\node at (-2.2,0) {$U_7$}; 

\draw[fill = white] ($(2)+(-90:0.4)$)circle(1.8mm); 
\node at ($(2)+(-90:0.4)$) {\tiny$a_1$};
\draw[fill = white] ($(3)+(-20:0.4)$)circle(1.8mm); 
\node at ($(3)+(-20:0.4)$) {\tiny$a_7$};
\draw[fill = white] ($(4)+(60:0.4)$)circle(1.8mm); 
\node at ($(4)+(60:0.4)$) {\tiny$a_6$};
\draw[fill = white] ($(5)+(180:0.8)$)circle(1.8mm); 
\node at ($(5)+(180:0.8)$) {\tiny$a_5$};
\draw[fill = white] ($(5)+(126:0.9)$)circle(1.8mm); 
\node at ($(5)+(126:0.9)$) {\tiny$a_4$};
\draw[fill = white] ($(5)+(72:0.8)$)circle(1.8mm); 
\node at ($(5)+(72:0.8)$) {\tiny$a_3$};
\draw[fill = white] ($(1)+(200:0.4)$)circle(1.8mm); 
\node at ($(1)+(200:0.4)$) {\tiny$a_2$};

\draw[fill = white] ($(1)+(-80:0.5)$)circle(1.8mm); 
\node at ($(1)+(-80:0.5)$) {\tiny$b$};
\draw[fill = white] ($(5)+(40:0.5)$)circle(1.8mm); 
\node at ($(5)+(40:0.5)$) {\tiny$\bar{b}$};

\draw[fill = white] ($(2)+(-8:0.6)$)circle(1.8mm); 
\node at ($(2)+(-8:0.6)$) {\tiny$c$};
\draw[fill = white] ($(1)+(112:0.5)$)circle(1.8mm); 
\node at ($(1)+(112:0.5)$) {\tiny$\bar{c}$};

\draw[fill = white] ($(2)+(165:0.7)$)circle(1.8mm); 
\node at ($(2)+(165:0.7)$) {\tiny$d$};
\draw[fill = white] ($(2)+(95:0.7)$)circle(1.8mm); 
\node at ($(2)+(95:0.7)$) {\tiny$\bar{d}$};

\draw[fill = white] ($(2)+(132:0.75)$)circle(1.7mm); 
\node at ($(2)+(132:0.75)$) {\tiny$e$};
\draw[fill = white] ($(2)+(27:0.7)$)circle(1.7mm); 
\node at ($(2)+(27:0.7)$) {\tiny$\bar{e}$};

\draw[fill = white] ($(3)+(120:0.4)$)circle(1.6mm); 
\node at ($(3)+(120:0.4)$) {\tiny$f$};
\draw[fill = white] ($(3)+(120:1)$)circle(1.6mm); 
\node at ($(3)+(120:1)$) {\tiny$\bar{f}$};

\draw[fill = white] ($(3)+(-120:0.4)$)circle(1.6mm); 
\node at ($(3)+(-120:0.4)$) {\tiny$g$};
\draw[fill = white] ($(3)+(-120:0.95)$)circle(1.6mm); 
\node at ($(3)+(-120:0.95)$) {\tiny$\bar{g}$};

\node at ($(2) + (200:0.7)$) {\small$w_1$};
\node at (1) [right]{\small$w_2$};
\node at (5) [below]{\small$w_3$};
\node at (4) [below]{\small$w_4$};
\node at (3) [left]{\small$w_5$};
\node at (3a) [left]{\small$w_6$};
\node at (3b) [left]{\small$w_7$};
\end{tikzpicture}
\end{equation*}
\end{example}

\subsection{Brauer configuration algebras}
Next, we recall the definition of Brauer configuration algebras \cite{GS17}. 
Our definition is slightly different from theirs because of Definition \ref{def:our BC}. 

\begin{definition}\label{def:BCA}
    Let $\Gamma := (H,\sigma,\psi,s,\m)$ be a Brauer configuration. Let $Q_{\Gamma}$ be a finite quiver defined as follows: 
    \begin{itemize}
        \item The set of vertices is the set $H/\psi$ of polygons of $\Gamma$. 
        \item The set of arrows is in bijection with the set $H$ of angles of $\Gamma$, 
        where we draw an arrow $[h]\to [\sigma(h)]$ for every $h\in H$. 
        We write this arrow by the same symbol $h$ or $\sigma^0(h)$ if there is no confusion. 
    \end{itemize}

For given $h,f\in H$ such that $f=\sigma^{m}(h)$ for some $1\leq m \leq \val(s(h))$, 
let $C_{h,f}$ be a path 
\begin{equation}\label{special cycle}
    C_{h,f} \colon [h] \xrightarrow{\sigma^0(h)} [\sigma(h)] \xrightarrow{\sigma(h)} [\sigma^2(h)] \longrightarrow \cdots \longrightarrow 
    [\sigma^{m-1}(h)] \xrightarrow{\sigma^{m-1}(h)} [\sigma^{m}(h)= f] 
\end{equation}
of length $m$ in the quiver $Q_{\Gamma}$. 
We have a cycle $C_{h} := C_{h,h}$ of length $\val(s(h))$. 
Let $I_{\Gamma}$ be an ideal in the path algebra $KQ_{\Gamma}$ generated by all the relations in (BC1) and (BC2): 

\begin{itemize}
    \item[(BC1)] For any polygon $V$ and any $h,f\in V$,   
    \begin{equation*}
        C_{h}^{\m(s(h))} - C_{f}^{\m(s(f))}.
    \end{equation*}
    \item[(BC2)] All paths in $Q_{\Gamma}$ 
    of length two which are not sub-paths of $C_{h}^{\m(s(h))}$ for any $h\in H$. 
\end{itemize}
    We define $A_{\Gamma}:=KQ_{\Gamma}/I_{\Gamma}$ and call it \emph{Brauer configuration algebra} of $\Gamma$. 
\end{definition}

In the above definition, for each $m$-gon $V$ of $\Gamma$, there are precisely $m$ arrows starting at $V$ and there are precisely $m$ arrows ending at $V$ in the quiver $Q_{\Gamma}$. In particular, $Q_{\Gamma}$ does not depend on the multiplicity function $\m$. 
On the other hand, the ideal $I_{\Gamma}$ is not admissible in general. In fact, $I_{\Gamma}$ may have elements of the form $\sigma^0(h) - C_{f}^{\m(s(f))}\in I_{\Gamma}$, where $h,f\in H$ are angles such that $s(h)$ is a truncated vertex and $s(f)$ is a non-truncated vertex. 
The Gabriel quiver of $A_{\Gamma}$ can be obtained from $Q_{\Gamma}$ by removing all such loops $\sigma^0(h)$ when $s(h)$ is a truncated vertex. 
Then, we find that our Brauer configuration algebras are naturally isomorphic to those defined in \cite[Definition 2.5]{GS17}. 
Notice that our definition includes the cases when $\Gamma$ consists of the single polygon all of whose vertices are truncated, where the corresponding Brauer configuration algebras are isomorphic to $K[x]/(x^2)$.
 
\begin{proposition}\cite[Theorem 4.1]{GS16}\label{prop:BCA is SMA}
\begin{enumerate}[\rm (a)]
    \item A Brauer configuration algebra is a symmetric special multiserial algebra. 
    \item A Brauer graph algebra is a Brauer configuration algebra all of whose polygons are edges. 
\end{enumerate} 
\end{proposition}

Let $A_{\Gamma}=KQ_{\Gamma}/I_{\Gamma}$ be the Brauer configuration algebra of a Brauer configuration $\Gamma$. 
By definition, there is a one-to-one correspondence between polygons $U$ and the (isomorphism classes of) indecomposable projective $A_{\Gamma}$-modules $P_U$. 
One can easily compute homomorphisms between two indecomposable projective modules as follows. 
Now, for two polygons $U,W$ of $\Gamma$, we set  
\begin{equation}    \label{eq:C_WU}
\mathcal{C}(W,U) := \left\{C_{h}^{r}C_{h,f} \,\middle|\, 
\begin{matrix}    
\text{
$h\in W$, $0\leq r\leq \m(s(h))-1$} \\  
\text{$f\in U$, $f = \sigma^t(h)$, $1\leq t \leq \val(s(h))$}
\end{matrix}\right\} 
\end{equation}
and denote by $\rho_{\alpha} \colon P_U\to P_W$ 
the homomorphism induced by a path $\alpha\in \mathcal{C}(W,U)$. 

\begin{lemma}\label{lem:Hom basis}
Let $U,W$ be polygons of $\Gamma$. 
Then, the set  
$\{\rho_{\alpha} \mid \alpha\in \mathcal{C}(W,U)\}$ forms a $K$-basis of the subspace of $\Hom(P_U,P_W)$ of radical maps from $P_U$ to $P_W$.
\end{lemma}

\begin{proof}
    It is obvious from (BC1) and (BC2). 
\end{proof}

\begin{example}
Let $\Gamma$ be a Brauer configuration given in Example \ref{example:BC}. 
We obtain a Brauer configuration algebra $A_{\Gamma} = kQ_{\Gamma}/I_{\Gamma}$, where $Q_{\Gamma}$ is the quiver given by 

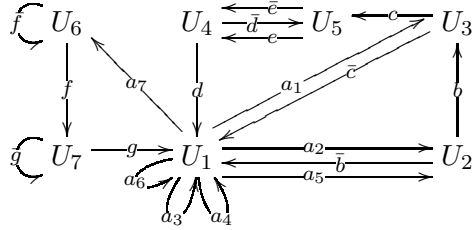
\begin{equation*}
\xymatrix@C=30pt@R=35pt{ 
 U_6 \ar@(ul,dl)|-{\bar{f}} 
 \ar@{->}[d]|-{f} & U_4 \ar@{->}[d]|-{d} \ar@{->}[r]|(.42){\bar{d}} & 
 U_5 \ar@{->}[l]<2mm>|(.42){e} \ar@{->}[l]<-2mm>|(.42){\bar{e}}& U_3 \ar@{->}[l]<-1mm>|-{c} 
 \ar@{->}[lld]<-0.7mm>|(.4){\bar{c}} 
  \\ 
 U_7 \ar@(ul,dl)|-{\bar{g}} 
 \ar@{->}[r]<1mm>|-{g} 
 & U_1 \ar@(l,ld)|-{a_6} \ar@(dl,d)|-{a_3} \ar@(d,dr)|-{a_4}
 \ar@{->}[lu]<-1mm>|-{a_7}
 \ar@{->}[rr]<1.2mm>|(.45){a_2} 
 \ar@{->}[rr]<-2.6mm>|(.45){a_5} 
 \ar@{->}[rru]<2.5mm>|(.4){a_1}
 && U_2 \ar@{->}[u]<0mm>|-{b} \ar@{->}[ll]<0.7mm>|(.45){\bar{b}} 
}
\end{equation*} 
and $I_{\Gamma}$ is the ideal generated by the following relations: 
\begin{itemize}
    \item $a_1c\bar{e}\bar{d}ed = a_2 b \bar{c} = a_3a_4a_5\bar{b} =a_4a_5\bar{b}a_3 = a_5\bar{b}a_3a_4 = a_6^2 = a_7fg$, 
    $b\bar{c}a_2 = \bar{b}a_3a_4a_5$, $c\bar{e}\bar{d}eda_1 = \bar{c}a_2b$, $da_1c\bar{e}\bar{d}e = \bar{d}eda_1c\bar{e}$,
    $eda_1c\bar{e}\bar{d} = \bar{e}\bar{d}eda_1c$, $fga_7=\bar{f}$ and $ga_7f = \bar{g}$. 
    \item All paths of length $2$ which are not sub-paths of the above monomials. 
\end{itemize}
In the following figure, we describe the indecomposable projective $A_{\Gamma}$-module $P_{V_i}$ corresponding to the polygon $V_i$ in term of its composition series, where each number $i$ shows the simple $A_{\Gamma}$-module at $V_i$. 
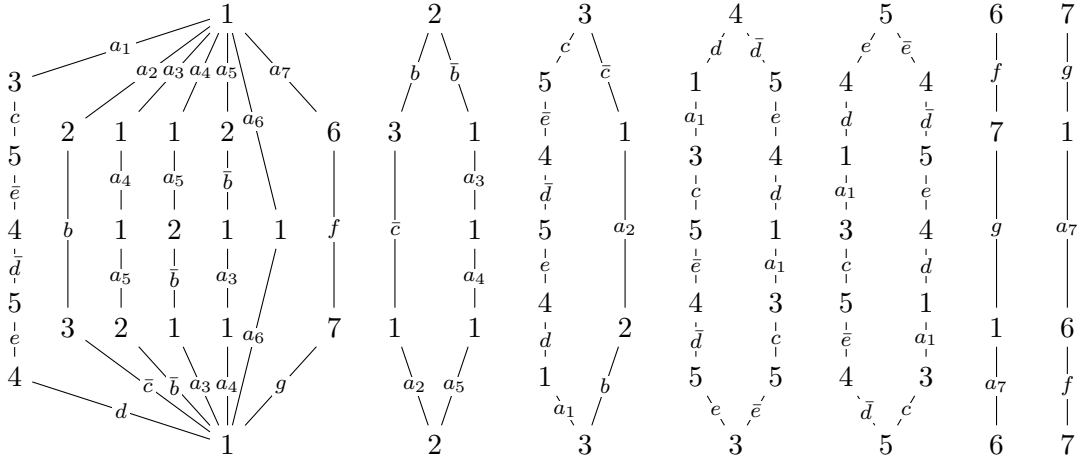
\begin{equation*}
    \begin{tabular}{ccccccccc}
    \begin{tikzpicture} [baseline=0mm]
    \coordinate (x) at (0.7,0) ;
    \coordinate (y) at (0,1.3) ;
    \node (top) at ($0*(x)+2.2*(y)$) {$1$};
    \node (q1)at ($-4*(x)+1.5*(y)$) {$3$};
    \node (q2)at ($-4*(x)+0.75*(y)$) {$5$};
    \node (q3)at ($-4*(x)+0*(y)$) {$4$};
    \node (q4)at ($-4*(x)+-0.75*(y)$) {$5$};
    \node (q5)at ($-4*(x)+-1.5*(y)$) {$4$};
    \node (w1)at ($-3*(x)+1*(y)$) {$2$};
    \node (w2)at ($-3*(x)+-1*(y)$) {$3$};
    \node (e1)at ($-2*(x)+1*(y)$) {$1$};
    \node (e2)at ($-2*(x)+0*(y)$) {$1$};
    \node (e3)at ($-2*(x)+-1*(y)$) {$2$};
    \node (r1)at ($-1*(x)+1*(y)$) {$1$};
    \node (r2)at ($-1*(x)+0*(y)$) {$2$};
    \node (r3)at ($-1*(x)+-1*(y)$) {$1$};
    \node (t1)at ($0*(x)+1*(y)$) {$2$};
    \node (t2)at ($0*(x)+0*(y)$) {$1$};
    \node (t3)at ($0*(x)+-1*(y)$) {$1$};
    \node (y1)at ($1*(x)+0*(y)$) {$1$};
    \node (u1)at ($2*(x)+1*(y)$) {$6$};
    \node (u2)at ($2*(x)+-1*(y)$) {$7$};
    \node (soc)at ($0*(x)+-2.2*(y)$) {$1$};
    \draw[-] (top)--node[fill=white, inner sep = 0.4mm]{\scriptsize$a_1$}(q1)--node[fill=white, inner sep = 0.4mm]{\scriptsize$c$}(q2)--node[fill=white, inner sep = 0.4mm]{\scriptsize$\bar{e}$}(q3)--node[fill=white, inner sep = 0.4mm]{\scriptsize$\bar{d}$}(q4)--node[fill=white, inner sep = 0.4mm]{\scriptsize$e$}(q5)--node[fill=white, inner sep = 0.4mm]{\scriptsize$d$}(soc);
    \draw[-] (top)--node[fill=white, inner sep = 0.4mm]{\scriptsize$a_2$}(w1)--node[fill=white, inner sep = 0.4mm]{\scriptsize$b$}(w2)--node[fill=white, inner sep = 0.4mm]{\scriptsize$\bar{c}$}(soc);
    \draw[-] (top)--node[fill=white, inner sep = 0.4mm]{\scriptsize$a_3$}(e1)--node[fill=white, inner sep = 0.4mm]{\scriptsize$a_4$}(e2)--node[fill=white, inner sep = 0.4mm]{\scriptsize$a_5$}(e3)--node[fill=white, inner sep = 0.4mm]{\scriptsize$\bar{b}$}(soc);
    \draw[-] (top)--node[fill=white, inner sep = 0.4mm]{\scriptsize$a_4$}(r1)--node[fill=white, inner sep = 0.4mm]{\scriptsize$a_5$}(r2)--node[fill=white, inner sep = 0.4mm]{\scriptsize$\bar{b}$}(r3)--node[fill=white, inner sep = 0.4mm]{\scriptsize$a_3$}(soc);
    \draw[-] (top)--node[fill=white, inner sep = 0.4mm]{\scriptsize$a_5$}(t1)--node[fill=white, inner sep = 0.4mm]{\scriptsize$\bar{b}$}(t2)--node[fill=white, inner sep = 0.4mm]{\scriptsize$a_3$}(t3)--node[fill=white, inner sep = 0.4mm]{\scriptsize$a_4$}(soc);
    \draw[-] (top)--node[fill=white, inner sep = 0.4mm]{\scriptsize$a_6$}(y1)--node[fill=white, inner sep = 0.4mm]{\scriptsize$a_6$}(soc); 
    \draw[-] (top)--node[fill=white, inner sep = 0.4mm]{\scriptsize$a_7$}(u1)--node[fill=white, inner sep = 0.4mm]{\scriptsize$f$}(u2)--node[fill=white, inner sep = 0.4mm]{\scriptsize$g$}(soc);
    \end{tikzpicture} & 
    \begin{tikzpicture} [baseline=0mm]
    \coordinate (x) at (0.7,0) ;
    \coordinate (y) at (0,1.3) ;
    \node (top) at ($0*(x)+2.2*(y)$) {$2$};
    \node (w1)at ($-0.75*(x)+1*(y)$) {$3$};
    \node (w2)at ($-0.75*(x)+-1*(y)$) {$1$};
    \node (e1)at ($0.75*(x)+1*(y)$) {$1$};
    \node (e2)at ($0.75*(x)+0*(y)$) {$1$};
    \node (e3)at ($0.75*(x)+-1*(y)$) {$1$};
    \node (soc) at ($0*(x)+-2.2*(y)$) {$2$};
    \draw[-] (top)--node[fill=white, inner sep = 0.4mm]{\scriptsize$b$}(w1)--node[fill=white, inner sep = 0.4mm]{\scriptsize$\bar{c}$}(w2)--node[fill=white, inner sep = 0.4mm]{\scriptsize$a_2$}(soc); 
    \draw[-] (top)--node[fill=white, inner sep = 0.4mm]{\scriptsize$\bar{b}$}(e1)--node[fill=white, inner sep = 0.4mm]{\scriptsize$a_3$}(e2)--node[fill=white, inner sep = 0.4mm]{\scriptsize$a_4$}(e3)--node[fill=white, inner sep = 0.4mm]{\scriptsize$a_5$}(soc);
    
    \end{tikzpicture} \ &
    \begin{tikzpicture} [baseline=0mm]
    \coordinate (x) at (0.7,0) ;
    \coordinate (y) at (0,1.3) ;
    \node (top) at ($0*(x)+2.2*(y)$) {$3$};
    \node (q1)at ($-0.75*(x)+1.5*(y)$) {$5$};
    \node (q2)at ($-0.75*(x)+0.75*(y)$) {$4$};
    \node (q3)at ($-0.75*(x)+0*(y)$) {$5$};
    \node (q4)at ($-0.75*(x)+-0.75*(y)$) {$4$};
    \node (q5)at ($-0.75*(x)+-1.5*(y)$) {$1$};
    \node (w1)at ($0.75*(x)+1*(y)$) {$1$};
    \node (w2)at ($0.75*(x)+-1*(y)$) {$2$};
    \node (soc) at ($0*(x)+-2.2*(y)$) {$3$};
    \draw[-] (top)--node[fill=white, inner sep = 0.4mm]{\scriptsize$c$}(q1)--node[fill=white, inner sep = 0.4mm]{\scriptsize$\bar{e}$}(q2)--node[fill=white, inner sep = 0.4mm]{\scriptsize$\bar{d}$}(q3)--node[fill=white, inner sep = 0.4mm]{\scriptsize$e$}(q4)--node[fill=white, inner sep = 0.4mm]{\scriptsize$d$}(q5)--node[fill=white, inner sep = 0.4mm]{\scriptsize$a_1$}(soc);
    \draw[-] (top)--node[fill=white, inner sep = 0.4mm]{\scriptsize$\bar{c}$}(w1)--node[fill=white, inner sep = 0.4mm]{\scriptsize$a_2$}(w2)--node[fill=white, inner sep = 0.4mm]{\scriptsize$b$}(soc);
    \end{tikzpicture} \ &
    \begin{tikzpicture} [baseline=0mm]
    \coordinate (x) at (0.7,0) ;
    \coordinate (y) at (0,1.3) ;
    \node (top) at ($0*(x)+2.2*(y)$) {$4$};
    \node (q1)at ($-0.75*(x)+1.5*(y)$) {$1$};
    \node (q2)at ($-0.75*(x)+0.75*(y)$) {$3$};
    \node (q3)at ($-0.75*(x)+0*(y)$) {$5$};
    \node (q4)at ($-0.75*(x)+-0.75*(y)$) {$4$};
    \node (q5)at ($-0.75*(x)+-1.5*(y)$) {$5$};
    \node (qq1)at ($0.75*(x)+1.5*(y)$) {$5$};
    \node (qq2)at ($0.75*(x)+0.75*(y)$) {$4$};
    \node (qq3)at ($0.75*(x)+0*(y)$) {$1$};
    \node (qq4)at ($0.75*(x)+-0.75*(y)$) {$3$};
    \node (qq5)at ($0.75*(x)+-1.5*(y)$) {$5$};
    \node (soc) at ($0*(x)+-2.2*(y)$) {$3$};
    \draw[-] (top)--node[fill=white, inner sep = 0.4mm]{\scriptsize$d$}(q1)--node[fill=white, inner sep = 0.4mm]{\scriptsize$a_1$}(q2)--node[fill=white, inner sep = 0.4mm]{\scriptsize$c$}(q3)--node[fill=white, inner sep = 0.4mm]{\scriptsize$\bar{e}$}(q4)--node[fill=white, inner sep = 0.4mm]{\scriptsize$\bar{d}$}(q5)--node[fill=white, inner sep = 0.4mm]{\scriptsize$e$}(soc);
    \draw[-] (top)--node[fill=white, inner sep = 0.4mm]{\scriptsize$\bar{d}$}(qq1)--node[fill=white, inner sep = 0.4mm]{\scriptsize$e$}(qq2)--node[fill=white, inner sep = 0.4mm]{\scriptsize$d$}(qq3)--node[fill=white, inner sep = 0.4mm]{\scriptsize$a_1$}(qq4)--node[fill=white, inner sep = 0.4mm]{\scriptsize$c$}(qq5)--node[fill=white, inner sep = 0.4mm]{\scriptsize$\bar{e}$}(soc);
    \end{tikzpicture} \ &
    \begin{tikzpicture} [baseline=0mm]
    \coordinate (x) at (0.7,0) ;
    \coordinate (y) at (0,1.3) ;
    \node (top) at ($0*(x)+2.2*(y)$) {$5$};
    \node (q1)at ($-0.75*(x)+1.5*(y)$) {$4$};
    \node (q2)at ($-0.75*(x)+0.75*(y)$) {$1$};
    \node (q3)at ($-0.75*(x)+0*(y)$) {$3$};
    \node (q4)at ($-0.75*(x)+-0.75*(y)$) {$5$};
    \node (q5)at ($-0.75*(x)+-1.5*(y)$) {$4$};
    \node (qq1)at ($0.75*(x)+1.5*(y)$) {$4$};
    \node (qq2)at ($0.75*(x)+0.75*(y)$) {$5$};
    \node (qq3)at ($0.75*(x)+0*(y)$) {$4$};
    \node (qq4)at ($0.75*(x)+-0.75*(y)$) {$1$};
    \node (qq5)at ($0.75*(x)+-1.5*(y)$) {$3$};
    \node (soc) at ($0*(x)+-2.2*(y)$) {$5$};
    \draw[-] (top)--node[fill=white, inner sep = 0.4mm]{\scriptsize$e$}(q1)--node[fill=white, inner sep = 0.4mm]{\scriptsize$d$}(q2)--node[fill=white, inner sep = 0.4mm]{\scriptsize$a_1$}(q3)--node[fill=white, inner sep = 0.4mm]{\scriptsize$c$}(q4)--node[fill=white, inner sep = 0.4mm]{\scriptsize$\bar{e}$}(q5)--node[fill=white, inner sep = 0.4mm]{\scriptsize$\bar{d}$}(soc);
    \draw[-] (top)--node[fill=white, inner sep = 0.4mm]{\scriptsize$\bar{e}$}(qq1)--node[fill=white, inner sep = 0.4mm]{\scriptsize$\bar{d}$}(qq2)--node[fill=white, inner sep = 0.4mm]{\scriptsize$e$}(qq3)--node[fill=white, inner sep = 0.4mm]{\scriptsize$d$}(qq4)--node[fill=white, inner sep = 0.4mm]{\scriptsize$a_1$}(qq5)--node[fill=white, inner sep = 0.4mm]{\scriptsize$c$}(soc);
    \end{tikzpicture} \ &
    \begin{tikzpicture} [baseline=0mm]
    \coordinate (x) at (0.7,0) ;
    \coordinate (y) at (0,1.3) ;
    \node (top) at ($0*(x)+2.2*(y)$) {$6$};
    \node (q1)at ($0*(x)+1*(y)$) {$7$};
    \node (q2)at ($0*(x)+-1*(y)$) {$1$};
    \node (soc) at ($0*(x)+-2.2*(y)$) {$6$};
    \draw[-] (top)--node[fill=white, inner sep = 0.4mm]{\scriptsize$f$}(q1)--node[fill=white, inner sep = 0.4mm]{\scriptsize$g$}(q2)--node[fill=white, inner sep = 0.4mm]{\scriptsize$a_7$}(soc);
    \end{tikzpicture} \ &
    \begin{tikzpicture} [baseline=0mm]
    \coordinate (x) at (0.7,0) ;
    \coordinate (y) at (0,1.3) ;
    \node (top) at ($0*(x)+2.2*(y)$) {$7$};
    \node (q1)at ($0*(x)+1*(y)$) {$1$};
    \node (q2)at ($0*(x)+-1*(y)$) {$6$};
    \node (soc) at ($0*(x)+-2.2*(y)$) {$7$};
    \draw[-] (top)--node[fill=white, inner sep = 0.4mm]{\scriptsize$g$}(q1)--node[fill=white, inner sep = 0.4mm]{\scriptsize$a_7$}(q2)--node[fill=white, inner sep = 0.4mm]{\scriptsize$f$}(soc);
    \end{tikzpicture}
    \end{tabular}
\end{equation*}
\end{example}

The next example Example \ref{example:BCAtilting}(ii) shows that a tilting mutation of Brauer configuration algebras is not a Brauer configuration algebra in general. 
As a consequence, the class of Brauer configuration algebras is not closed under derived equivalence. 

\begin{example}\label{example:BCAtilting}
Let $A_{\Gamma}$ be a Brauer configuration algebra whose Brauer configuration $\Gamma$ is multiplicity-free and given by the following figure. 
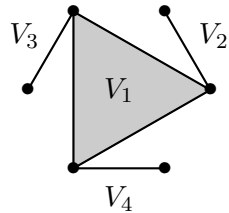
\begin{equation*}    
\begin{tikzpicture}[baseline=0mm,scale = 0.8]
    \coordinate (0) at (360:1.5);
    \coordinate (1) at (60:1.5); 
    \coordinate (2) at (120:1.5);
    \coordinate (3) at (180:1.5); 
    \coordinate (4) at (240:1.5);
    \coordinate (5) at (300:1.5); 
    \coordinate (t1) at ($(0)!0.5!(2)$);
    \draw[line width =0mm, fill = white!20!lightgray] (0)--(2)--(4)--cycle;
    \draw[thick] (0)--(1) (2)--(3) (4)--(5); 
    
    \draw[thick] (0)--(2)--(4)--cycle;
    \node at (0) {$\bullet$};
    \node at (1) {$\bullet$};
    \node at (2) {$\bullet$};
    \node at (3) {$\bullet$};
    \node at (4) {$\bullet$};
    \node at (5) {$\bullet$};
    \node at (0,0) {$V_1$};
    \node at ($($(0)!0.5!(1)$)+(30:0.5)$) {$V_2$};
    \node at ($($(2)!0.5!(3)$)+(150:0.5)$) {$V_3$};
    \node at ($($(4)!0.5!(5)$)+(-90:0.5)$) {$V_4$};
\end{tikzpicture}
\end{equation*}
Then, it is isomorphic to the trivial extension of the path algebra of type $D_4$ with a minimal co-generator. 
For each $i\in \{1,2,3,4\}$, we compute tilting mutation of $A_{\Gamma}$ at the projective module $P_{V_i}$ as follows.  
\begin{enumerate}[\rm (i)]
    \item For $i=1$, we have  $\End(\mu^-_{P_{V_{1}}}(A_{\Gamma}))\cong A_{\Gamma}$ as $K$-algebras by a direct calculation. 
    This also follows from Theorem \ref{thm:main theorem} by using flip (see Section \ref{sec:main result} for the detail). 
    \item For $i \in \{2,3,4\}$, the endomorphism algebra 
    $\End(\mu_{P_{V_i}}^-(A_{\Gamma}))$ is isomorphic to a bound quiver algebra given by the following quiver with relations:  
\begin{equation*}
\begin{tikzpicture}
    \node at (0,0) {$\xymatrix@C=26pt@R=10pt{ 
&3\ar@{->}[rd]^{a_2}& \\ 
1 \ar@{->}[ru]^{a_1}\ar@{->}[rd]_{b_1} &&2 \ar@{->}[ll]_{c}\\ 
&4\ar@{->}[ru]_{b_2}&
}$};
\node at (2.5,0) {\text{and}};
\node at (5.5,0) {$
 \left< \begin{matrix} a_1a_2-b_1b_2, a_2cb_1, \\ b_2ca_1, ca_1a_2c \end{matrix} \right>.$};
\end{tikzpicture}
\end{equation*}
Since it is not multiserial, it can not be realized as a Brauer configuration algebra. 
\end{enumerate}
\end{example}

\section{Flip of Brauer configurations and main result}
\label{sec:main result}
An aim of this section is to introduce a combinatorial operation on Brauer configurations, called \emph{flip}, and prove our main result (Theorem \ref{thm:main theorem}). 

\subsection{Definition of flip}
\label{subsec:flip BC}
We will define flip of Brauer configurations under the following assumption on polygons.
Let $\Gamma = (H,\sigma,\psi,s,\m)$ be a Brauer configuration. 

\begin{definition}\label{def:conditionE}
We say that a polygon $V$ of $\Gamma$ satisfies the condition {\rm (E)} if, for every $e\in V$, the polygon $[\sigma^{-1}(e)]$ is an edge or $V$ itself. 
\end{definition}

We use the following notations. 
\begin{definition}
Let $V$ be a polygon of $\Gamma$. 
\begin{enumerate}[\rm (1)]
\item We decompose $V$ into a disjoint union of the following subsets: 
\begin{eqnarray*}
    \HH_1(V) &:=& \{e\in V \mid \text{$s^{-1}(s(e))\subset V$}\},\\ 
    \HH_2(V) &:=& \{e\in V \mid \text{$s^{-1}(s(e))\not\subset V$, $\sigma(e)\in V$}\}, \\
    \HH_3(V) &:=& \{e\in V \mid \text{$s^{-1}(s(e))\not\subset V$, $\sigma(e)\not\in V$}\}. 
\end{eqnarray*}
    \item For each $f\in H\setminus \HH_1(V)$, let 
\begin{equation*}
\p_V(f) := \sigma^{-c}(f) \quad \text{and} \quad 
\n_V(f) := \sigma^d(f), 
\end{equation*}
where 
$c:=\min\{i\geq 1 \mid \sigma^{-i}(f)\not\in V\}$ and
$d:=\min\{j \geq 1 \mid \sigma^{j}(f)\not\in V\}$. 
Thus, $\p_V(f),\n_V(f) \not\in V$. If $V$ is clear, we will omit it. 
    \item If $V$ satisfies the condition (E), then we set  
\begin{eqnarray*}
    \HH_4(V) &:=& \{f \in  H\setminus V \mid \text{$\exists \,e \in V$ s.t. $\n(f)=\overline{\sigma^{-1}(e)}$}\} \quad \text{and} \\
    \HH_5(V) &:=& \{f \in  H\setminus V \mid \text{$\forall\,e\in V, \n(f)\neq\overline{\sigma^{-1}(e)}$}\}. 
\end{eqnarray*}
We notice that, for any $e\in V$ such that $\sigma^{-1}(e)\not\in V$, 
the polygon $[\sigma^{-1}(e)]$ is an edge by the condition (E),  
so we write $[\sigma^{-1}(e)] = \{\sigma^{-1}(e),\overline{\sigma^{-1}(e)}\}$ (see Definition \ref{def:basicBC}(4)).
In addition, for each $f\in \HH_4(V)$, there exists a unique angle $e\in V$, which we will denote by $e=x_f$, satisfying $\n(f)= \overline{\sigma^{-1}(e)}$ in the definition of $\HH_4(V)$.

Thus, we obtain a decomposition $H = \bigsqcup_{i=1}^5\HH_i(V)$ associated to the polygon $V$ satisfying the condition (E). 
\end{enumerate}

\end{definition}

\begin{definition}\label{def:flip BC}
Let $\Gamma = (H,\sigma,\psi,s,\m)$ be a Brauer configuration. 
For a polygon $V$ satisfying the condition (E), we define a new Brauer configuration $\Gamma' = (H',\sigma',\psi',s',\m')$ in the following way. 
We set $H':=H$ and $\psi' := \psi$. 
In addition, we define an angle $\sigma'(h)$ for all $h\in H = \bigsqcup_{i=1}^5\HH_i(V)$ by rules in Table \ref{table:flip}. 
Notice that, by the condition (E), for any $e\in \HH_2(V)\sqcup \HH_3(V)$, the polygon $[\p(e)]$ is an edge, so we write $[\p(e)] = \{\p(e),\overline{\p(e)}\}$. 

\begin{table}[h]
\renewcommand{\arraystretch}{1.4}
\begin{tabular}{c||c|c|c|c|cccc}
{\rm Cases } & {\rm Angle $h$} & $\sigma'(h)$ & $s'(h)$\\ \hline \hline 
{\rm (1)} & $e\in \HH_1(V)$ & $\sigma(e)$ & $s(e)$ \\  \hline  
{\rm (2)} & $e\in \HH_2(V)$ & $\sigma(e)$ & $s(\overline{\p(e)})$ \\ 
{\rm (3)} & $e\in \HH_3(V)$ & $\overline{\p(e)}$ & $s(\overline{\p(e)})$\\ \hline 
{\rm (4)} & $f\in \HH_4(V)$ & $x_f$ & $s(f)$\\
{\rm (5)}  & $f\in \HH_5(V)$& $\n(f)$ & $s(f)$ \\ 
\end{tabular}
\vspace{2mm}
    \caption{The definition of $\sigma'(h)$.}
    \label{table:flip}
\end{table}

Then, $\sigma'$ gives rise to a permutation on $H'$. Furthermore, a canonical surjection 
$s'\colon H' \to H'/ \langle \sigma' \rangle$ induces a one-to-one correspondence $\gamma\colon H'/ \langle \sigma' \rangle \xrightarrow{\sim} H/ \langle \sigma\rangle$ by mapping $s'(h)$ to the corresponding vertex, say $\gamma(s'(h))$, of $\Gamma$ in the table.  
Then, we define a multiplicity function $\m':=\m\circ \gamma \colon H'/\langle \sigma'\rangle \to \mathbb{Z}_{>0}$ by a composition, that is, 
\begin{equation}\label{new m}
\m'(s'(h)) :=   
\begin{cases}
    \m(s(h)) & \text{($h\in \HH_1(V)$),} \\ 
    \m(s(\overline{\p(h)}) & \text{($h\in \HH_2(V)\sqcup \HH_3(V)$),} \\ 
    \m(s(h)) & \text{($h\in \HH_4(V)\sqcup \HH_5(V)$).}
\end{cases} 
\end{equation}

We call $\Gamma'$ a \emph{left flip} of $\Gamma$ at $V$ and write $\mu_V^-(\Gamma) = \Gamma'$.
Dually, we can define a \emph{right flip} $\mu_V^+(\Gamma)$  under the dual of the condition of (E) on $V$, i.e., for every $h\in V$ the polygon $[\sigma(h)]$ is an edge or $V$ itself. 
\end{definition}

\begin{example}  
Let $\Gamma$ be a Brauer configuration given in Example \ref{example:BC}. 
In this case, the polygon $U:= U_1$ satisfies the condition (E). 
Moreover, a left flip of $\Gamma$ at this polygon is given by the following figure. 
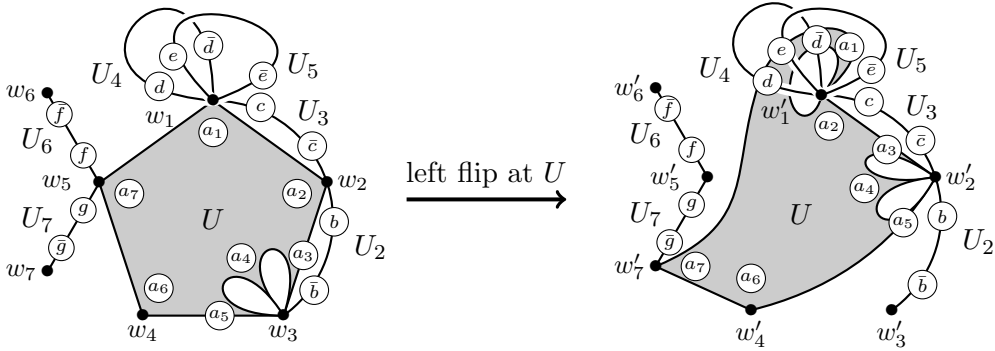
\vspace{-1.5cm}
\begin{equation*} 
\begin{tikzpicture}[baseline = 0cm]
\node at (0,0) {$
\begin{tikzpicture}[baseline=0mm, scale =1.05]
\coordinate (1) at (90-72:1.5);
\coordinate (2) at (90:1.5); 
\coordinate (3) at (90+72*1:1.5);
\coordinate (4) at (90+72*2:1.5); 
\coordinate (5) at (90+72*3:1.5);
\coordinate (3a) at ($(3)+(120:1.3)$);
\coordinate (3b) at ($(3)+(-120:1.3)$);
\draw[thick] (2)arc(90:90-72*2:1.5);
\fill[line width =0mm, fill = white!20!lightgray] (1)--(2)--(3)--(4)--(5)--cycle; 
\draw[thick] (1)--(2)--(3)--(4)--(5)--cycle; 

\draw[fill=white, line width =0mm, thick] (5)..controls($(5)+(72:1.3)$)and($(5)+(72+108-50:1.3)$)..(5);
\draw[fill=white, line width =0mm, thick] (5)..controls($(5)+(72+50:1.3)$)and($(5)+(72+108:1.3)$)..(5);

\draw[thick] (3)--(3a); 
\draw[thick] (3)--(3b);

\draw[white, line width = 1.5mm] (2)..controls($(2)+(90:2.5)$)and($(2)+(90+85:2.5)$)..(2);
\draw[thick] (2)..controls($(2)+(90:2.5)$)and($(2)+(90+85:2.5)$)..(2);
\draw[white, line width = 1.5mm] (2)..controls($(2)+(90-70:2.7)$)and($(2)+(90+45:2.7)$)..(2);
\draw[thick] (2)..controls($(2)+(90-70:2.7)$)and($(2)+(90+45:2.7)$)..(2);

\node at(1) {$\bullet$}; 
\node at(2) {$\bullet$}; 
\node at(3) {$\bullet$}; 
\node at(4) {$\bullet$}; 
\node at(5) {$\bullet$}; 
\node at(3a) {$\bullet$};
\node at(3b) {$\bullet$};
\node at (0,0) {$U$}; 
\node at (-10:2) {$U_2$}; 
\node at (47:1.85) {$U_3$}; 
\node at ($(124:2.2)+(180:0.1)$) {$U_4$}; 
\node at ($(62:2.2)+(0:0.1)$) {$U_5$}; 
\node at (-2.2,1) {$U_6$}; 
\node at (-2.2,0) {$U_7$}; 

\draw[fill = white] ($(2)+(-90:0.4)$)circle(1.8mm); 
\node at ($(2)+(-90:0.4)$) {\tiny$a_1$};
\draw[fill = white] ($(3)+(-20:0.4)$)circle(1.8mm); 
\node at ($(3)+(-20:0.4)$) {\tiny$a_7$};
\draw[fill = white] ($(4)+(60:0.4)$)circle(1.8mm); 
\node at ($(4)+(60:0.4)$) {\tiny$a_6$};
\draw[fill = white] ($(5)+(180:0.8)$)circle(1.8mm); 
\node at ($(5)+(180:0.8)$) {\tiny$a_5$};
\draw[fill = white] ($(5)+(126:0.9)$)circle(1.8mm); 
\node at ($(5)+(126:0.9)$) {\tiny$a_4$};
\draw[fill = white] ($(5)+(72:0.8)$)circle(1.8mm); 
\node at ($(5)+(72:0.8)$) {\tiny$a_3$};
\draw[fill = white] ($(1)+(200:0.4)$)circle(1.8mm); 
\node at ($(1)+(200:0.4)$) {\tiny$a_2$};

\draw[fill = white] ($(1)+(-80:0.5)$)circle(1.8mm); 
\node at ($(1)+(-80:0.5)$) {\tiny$b$};
\draw[fill = white] ($(5)+(40:0.5)$)circle(1.8mm); 
\node at ($(5)+(40:0.5)$) {\tiny$\bar{b}$};

\draw[fill = white] ($(2)+(-8:0.6)$)circle(1.8mm); 
\node at ($(2)+(-8:0.6)$) {\tiny$c$};
\draw[fill = white] ($(1)+(112:0.5)$)circle(1.8mm); 
\node at ($(1)+(112:0.5)$) {\tiny$\bar{c}$};

\draw[fill = white] ($(2)+(165:0.7)$)circle(1.8mm); 
\node at ($(2)+(165:0.7)$) {\tiny$d$};
\draw[fill = white] ($(2)+(95:0.7)$)circle(1.8mm); 
\node at ($(2)+(95:0.7)$) {\tiny$\bar{d}$};

\draw[fill = white] ($(2)+(132:0.75)$)circle(1.7mm); 
\node at ($(2)+(132:0.75)$) {\tiny$e$};
\draw[fill = white] ($(2)+(27:0.7)$)circle(1.7mm); 
\node at ($(2)+(27:0.7)$) {\tiny$\bar{e}$};

\draw[fill = white] ($(3)+(120:0.4)$)circle(1.6mm); 
\node at ($(3)+(120:0.4)$) {\tiny$f$};
\draw[fill = white] ($(3)+(120:1)$)circle(1.6mm); 
\node at ($(3)+(120:1)$) {\tiny$\bar{f}$};

\draw[fill = white] ($(3)+(-120:0.4)$)circle(1.6mm); 
\node at ($(3)+(-120:0.4)$) {\tiny$g$};
\draw[fill = white] ($(3)+(-120:0.95)$)circle(1.6mm); 
\node at ($(3)+(-120:0.95)$) {\tiny$\bar{g}$};

\node at ($(2) + (200:0.7)$) {\small$w_1$};
\node at (1) [right]{\small$w_2$};
\node at (5) [below]{\small$w_3$};
\node at (4) [below]{\small$w_4$};
\node at ($(3)+(180:0.2)$) [left]{\small$w_5$};
\node at (3a) [left]{\small$w_6$};
\node at (3b) [left]{\small$w_7$};
\end{tikzpicture}$};  
\node at (8,0) {$
\begin{tikzpicture}[scale = 1.05] 
\coordinate (1) at (90-72:1.5);
\coordinate (2) at (90:1.5); 
\coordinate (3) at (90+72*1:1.5);
\coordinate (4) at (90+72*2:1.5); 
\coordinate (5) at (90+72*3:1.5);
\coordinate (3a) at ($(3)+(120:1.3)$);
\coordinate (3b) at ($(3)+(-120:1.3)$);
\draw[thick] (2)arc(90:90-72*2:1.5);
\fill[line width =0mm, fill = white!20!lightgray] (1)--(2)..controls($(2)+(90+-50:1.5)$)and($(2)+(90+20:1.3)$)..($(2)+(170:0.8)$)..controls($(2)+(220:1.5)$)and($(3b)+(23:1.4)$)..(3b)--(4)..controls(-90:1)and(-30:1)..cycle; 
\draw[fill=white, line width =0mm, thick] (2)..controls($(2)+(90-40:1)$)and($(2)+(90+25:1)$)..($(2)+(180:0.4)$)..controls($(2)+(270+10:0.4)+(180:0.4)$)and($(2)+(180+80:0.3)$)..(2);
\draw[thick] (1)--(2)..controls($(2)+(90+-50:1.5)$)and($(2)+(90+20:1.3)$)..
($(2)+(170:0.8)$)..controls($(2)+(220:1.5)$)and($(3b)+(23:1.4)$)..(3b)--(4)..controls(-90:1)and(-30:1)..cycle;

\draw[fill=white, line width =0mm, thick] (1)..controls($(1)+(144:1.3)$)and($(1)+(180+10:1.3)$)..(1);
\draw[fill=white, line width =0mm, thick] (1)..controls($(1)+(180+5:1.3)$)and($(1)+(-108-10:1.3)$)..(1);
\draw[thick] (3)--(3a) (3)--(3b);

\draw[white, line width = 1.5mm] (2)..controls($(2)+(90:2.5)$)and($(2)+(90+85:2.5)$)..(2);
\draw[thick] (2)..controls($(2)+(90:2.5)$)and($(2)+(90+85:2.5)$)..(2);
\draw[white, line width = 1.5mm] (2)..controls($(2)+(90-70:2.7)$)and($(2)+(90+45:2.7)$)..(2);
\draw[thick] (2)..controls($(2)+(90-70:2.7)$)and($(2)+(90+45:2.7)$)..(2);

\node at(1) {$\bullet$}; 
\node at(2) {$\bullet$}; 
\node at(3) {$\bullet$}; 
\node at(4) {$\bullet$}; 
\node at(5) {$\bullet$}; 
\node at(3a) {$\bullet$};
\node at(3b) {$\bullet$};

\draw[fill = white] ($(2)+(60:0.7)$)circle(1.8mm); 
\node at ($(2)+(60:0.7)$) {\tiny$a_1$};
\draw[fill = white] ($(3b)+(0:0.5)$)circle(1.8mm); 
\node at ($(3b)+(0:0.5)$) {\tiny$a_7$};
\draw[fill = white] ($(4)+(90:0.4)$)circle(1.8mm); 
\node at ($(4)+(90:0.4)$) {\tiny$a_6$};
\draw[fill = white] ($(1)+(180+55:0.7)$)circle(1.8mm); 
\node at ($(1)+(180+55:0.7)$) {\tiny$a_5$};
\draw[fill = white] ($(1)+(126+63:0.9)$)circle(1.8mm); 
\node at ($(1)+(126+63:0.9)$) {\tiny$a_4$};
\draw[fill = white] ($(1)+(72+78:0.7)$)circle(1.8mm); 
\node at ($(1)+(72+78:0.7)$) {\tiny$a_3$};
\draw[fill = white] ($(2)+(-76:0.4)$)circle(1.8mm); 
\node at ($(2)+(-76:0.4)$) {\tiny$a_2$};

\draw[fill = white] ($(1)+(-80:0.5)$)circle(1.8mm); 
\node at ($(1)+(-80:0.5)$) {\tiny$b$};
\draw[fill = white] ($(5)+(40:0.5)$)circle(1.8mm); 
\node at ($(5)+(40:0.5)$) {\tiny$\bar{b}$};

\draw[fill = white] ($(2)+(-8:0.6)$)circle(1.8mm); 
\node at ($(2)+(-8:0.6)$) {\tiny$c$};
\draw[fill = white] ($(1)+(112:0.5)$)circle(1.8mm); 
\node at ($(1)+(112:0.5)$) {\tiny$\bar{c}$};

\draw[fill = white] ($(2)+(165:0.7)$)circle(1.8mm); 
\node at ($(2)+(165:0.7)$) {\tiny$d$};
\draw[fill = white] ($(2)+(95:0.7)$)circle(1.8mm); 
\node at ($(2)+(95:0.7)$) {\tiny$\bar{d}$};

\draw[fill = white] ($(2)+(132:0.75)$)circle(1.7mm); 
\node at ($(2)+(132:0.75)$) {\tiny$e$};
\draw[fill = white] ($(2)+(27:0.7)$)circle(1.7mm); 
\node at ($(2)+(27:0.7)$) {\tiny$\bar{e}$};

\draw[fill = white] ($(3)+(120:0.4)$)circle(1.6mm); 
\node at ($(3)+(120:0.4)$) {\tiny$f$};
\draw[fill = white] ($(3)+(120:1)$)circle(1.6mm); 
\node at ($(3)+(120:1)$) {\tiny$\bar{f}$};

\draw[fill = white] ($(3)+(-120:0.4)$)circle(1.6mm); 
\node at ($(3)+(-120:0.4)$) {\tiny$g$};
\draw[fill = white] ($(3)+(-120:0.95)$)circle(1.6mm); 
\node at ($(3)+(-120:0.95)$) {\tiny$\bar{g}$};

\node at (-0.25,0) {$U$}; 
\node at (-10:2) {$U_2$}; 
\node at (47:1.85) {$U_3$}; 
\node at ($(124:2.2)+(180:0.1)$) {$U_4$}; 
\node at ($(62:2.2)+(0:0.1)$) {$U_5$}; 
\node at (-2.2,1) {$U_6$}; 
\node at (-2.2,0) {$U_7$};

\node at ($(2) + (200:0.6)$) {\small$w_1'$};
\node at (1) [right]{\small$w_2'$};
\node at (5) [below]{\small$w_3'$};
\node at (4) [below]{\small$w_4'$};
\node at ($(3)+(180:0.2)$) [left]{\small$w_5'$};
\node at (3a) [left]{\small$w_6'$};
\node at (3b) [left]{\small$w_7'$};
\end{tikzpicture}$}; 
\draw[->, fill=white,line width = 1.5pt] (2.6,-1)--node[above]{\text{left flip at $U$}}(4.7,-1);
\end{tikzpicture}
\end{equation*}
In fact, we can check that 
$\mathsf{H}_1(U) = \{a_6\}$, 
$\mathsf{H}_2(U) = \{a_3,a_4\}$ and 
$\mathsf{H}_3(U) = \{a_1,a_2,a_5,a_7\}$. 
In addition, $\mathsf{H}_4(U) = \{\bar{c}, d, \bar{e},\bar{g}\}$ and 
$\mathsf{H}_5(U) = \{b, \bar{b}, c, \bar{d}, e, f,\bar{f}, g\}$, where we have $x_{\bar{c}}=a_3$, $x_d=a_2$, $x_{\bar{e}}= a_1$ and $x_{\bar{g}}=a_7$. 
Therefore, by Table \ref{table:flip}, 
a new permutation $\sigma'$ on $H$ is provided by 
\begin{gather*} 
    w_1' := (a_1,\bar{d},e,d,a_2,c,\bar{e}), \ 
    w_2' := (a_3,a_4,a_5, b,\bar{c}), \
    w_3' := (\bar{b}), \\ 
    w_4' := (a_6), \ 
    w_5' := (f,g), \ 
    w_6' := (\bar{f}), \ 
    w_7' := (a_7,\bar{g}). 
\end{gather*}
Since the induced one-to-one correspondence $\gamma$ of vertices is given by mapping $w_i\mapsto w_i'$ for $i\in \{1,\ldots,6\}$, 
the multiplicity of $w_i'$ is the same as that of $w_i$ for every $i$.  
\end{example}

\subsection{Main result}\label{subsec:main result}

The following is a main result in this paper.

\begin{theorem}\label{thm:main theorem}
Let $\Gamma$ be a Brauer configuration and $A_{\Gamma}$ the Brauer configuration algebra of $\Gamma$ over an algebraically closed field $K$. 
If a polygon $V$ of $\Gamma$ satisfies the condition {\rm (E)}, then there is an isomorphism of $K$-algebras
\begin{equation*}
    \End(\mu_{P_{V}}^{-}(A_{\Gamma}))\cong A_{\mu_V^-(\Gamma)}. 
\end{equation*}
In particular, it is again a Brauer configuration algebra. 
\end{theorem}

When $\Gamma$ is a Brauer graph, the condition (E) is automatically satisfied for all edges, and our flip coincides with that of Brauer graphs. Thus, \cite[Theorem 1.3]{Aihara15} is a special case of our result. 

In the rest, we prove Theorem \ref{thm:main theorem}. 
Keep the notation in the statement, 
we recall that the Brauer configuration algebra of $\Gamma':=\mu_{P_V}^-(\Gamma)$ is defined to be a factor algebra 
$$A_{\Gamma'} := KQ_{\Gamma'}/I_{\Gamma'}.$$ 

Our strategy of a proof of Theorem \ref{thm:main theorem} is the following. 
\begin{itemize}
    \item In Subsection \ref{subsub:construction phi}, we will construct a homomorphism $\varphi \colon KQ_{\Gamma'} \to \End(T)$ of $K$-algebras, 
    where $T\cong \mu_{P_V}^-(A_{\Gamma})$. 
    \item In Subsection \ref{subsub:the induced map}, we show that it factors through the natural surjection $\pi\colon KQ_{\Gamma'}\to A_{\Gamma'}$ and induces the homomorphism $\bar{\varphi}\colon A_{\Gamma'}\to \End(T)$ such that $\varphi = \bar{\varphi}\circ \pi$. 
    \[
    \xymatrix{ KQ_{\Gamma'} \ar[rr]^{\pi} \ar[rd]_{\varphi} \ar@{}[rrd]|{\circlearrowright} & & A_{\Gamma'} \ar@{.>}[ld]^{\overline{\varphi}} \\
    & \End(T) &  }
    \]
    \item In Subsection \ref{subsub:isomorphism}, we prove that 
    $\bar{\varphi}$ is an isomorphism. It completes a proof. 
\end{itemize}

\subsubsection{A construction of $\varphi$}\label{subsub:construction phi}

Let $\Gamma:=(H,\sigma,\psi,s,\mathfrak{m})$ be a Brauer configuration and $A_{\Gamma}=KQ_{\Gamma}/I_{\Gamma}$ the Brauer configuration algebra corresponding to $\Gamma$. 
We first compute a left mutation of $A_{\Gamma}$ with respect to $P_V$. 

\begin{proposition}\label{prop:approximation of PV}
Let $V$ be a polygon of $\Gamma$ (not necessarily satisfying the condition (E)). 
Then, a left mutation $\mu^-_{P_V}(A_{\Gamma})$ is isomorphic to a direct sum $T = \bigoplus_{U\in H/\psi} T_U$, where $T_U=P_U$ for all $U\neq V$, and $T_V$ is a two-term complex $T_V=(T_V^{-1}\overset{d_V}{\longrightarrow} T_V^0)$ given by
\begin{equation*}
    T_V^{0} := \bigoplus_{e\in \HH_{2}(V)\sqcup \HH_{3}(V)} P_{[\p(e)]}, 
    \quad 
    T_V^{-1} := P_V \quad \text{and} \quad d_V := \sum_{e\in \HH_{2}(V)\sqcup \HH_{3}(V)} \iota_{e} \rho_{C_{\p(e),e}}. 
\end{equation*}
In the above, $\iota_e \colon P_{[\p(e)]} \to T_V^{0}$ is a canonical inclusion and $\rho_{C_{\p(e),e}}\colon P_V = P_{[e]} \to P_{[\p(e)]}$ is the homomorphism induced by the path $C_{\p(e),e}$ of the quiver $Q_{\Gamma}$ for each $e\in \HH_{2}(V)\sqcup \HH_{3}(V)$.
\end{proposition}

\begin{proof}
We need to show that the map $d_V$ is a left minimal $(\add \bigoplus_{U\neq V}P_U)$-approximation of $P_V$. 
For each $U\neq V$, the space $\Hom(P_V,P_U)$ is generated by paths in $\mathcal{C}(U,V)$ by Lemma \ref{lem:Hom basis}. 
In particular, every non-zero path from $U$ to $V$ factors through at least one paths of the form $C_{\p(e),e}$ with $e\in \HH_{2}(V)\sqcup \HH_{3}(V)$. 
This shows that $d_V$ gives a left $(\add \bigoplus_{U\neq V}P_U)$-approximation of $P_V$. 
Furthermore, it is left minimal from this description. 
Thus, we get the assertion. 
\end{proof}

From now on, we fix a polygon $V$ of $\Gamma$ and assume that it satisfies the condition (E). 
We have a decomposition $H=\bigsqcup_{i=1}^5\HH_i(V)$ associated to $V$.
In the rest, let $T = \bigoplus_{U\in H/\psi} T_U$ be a two-term complex given in Proposition \ref{prop:approximation of PV} with respect to $V$, which is isomorphic to a left mutation $\mu_{P_V}^-(A_{\Gamma})$. 
In addition, let $\Gamma':= \mu^-_{V}(\Gamma) = (H',\sigma',\psi',s',\m')$ be a left flip of $\Gamma$ at $V$ 
and $A_{\Gamma'} = KQ_{\Gamma'}/I_{\Gamma'}$ the Brauer configuration algebra of $\Gamma'$. 
By definition, we have $H'=H$ and $\psi' = \psi$. 
The set of vertices of $Q_{\Gamma'}$ is the set $H'/\psi' = H/\psi$ of polygons, 
and the set of arrows of $Q_{\Gamma'}$ is given by 
the set $\{\sigma'^0(h)\colon [h] \to[\sigma'(h)] \mid h\in H'\}$. 

We show the next claim. 

\begin{proposition}\label{construction of varphi}
    We have a homomorphism $\varphi \colon KQ_{\Gamma'} \to \End(T)$ of $K$-algebras. 
\end{proposition}

\begin{proof}
To construct a homomorphism from the path algebra $KQ_{\Gamma'}$ to the endomorphism algebra $\End(T)$, it suffices to associate each arrow of $Q_{\Gamma'}$ to a homomorphism between indecomposable direct summands of $T=\bigoplus_{U\in H/\psi} T_U$. 

The set of arrows of $Q_{\Gamma'}$ bijectively corresponds to the set $H'=H$. 
For each angle $h\in H = \bigsqcup_{i=1}^{5} \HH_i(V)$, 
we define a map
$$\varphi_1(\sigma'^0(h)) \colon T_{[\sigma'(h)]}\to T_{[h]}
$$ 
by rules in Table \ref{table:hom}, where $\sigma'(h)$ is the angle defined by Table \ref{table:flip}.
Here, for $e\in \HH_2(V)\sqcup \HH_3(V)$, let $\iota_e \colon P_{[\p(e)]} \to T_V^0$ and $\pi_e \colon T_V^0 \to P_{[\p(e)]}$ be a canonical inclusion and canonical projection respectively. 
In addition, for two $e,e'\in \HH_2(V)\sqcup \HH_3(V)$ such that $\p(e)=\p(e')$, 
let $\eta_{e,e'}\colon T_V^0\to T_V^0$ be a composition 
\begin{equation*}
    \eta_{e,e'} \colon T_V^0 \xrightarrow{\pi_{e'}} P_{[\p(e')]} \xrightarrow{1}  P_{[\p(e)]} \xrightarrow{\iota_{e}} T_V^0.
\end{equation*}
Finally, let $\zeta_{s(e)}\in K$ be an element $z$ satisfying the equation $z^{\m(s(e))\val(s(e)))} = -1$ for every $e\in V$ (It exists since $K$ is an algebraically closed field). 

\begin{table}[h]
\renewcommand{\arraystretch}{1.4}
    \begin{tabular}{c||c|c|c|c|ll}
    \renewcommand{\arraystretch}{1.4}
       {\rm Cases} & {\rm Angle $h$} & $\sigma'(h)$ & $T_{[\sigma'(h)]}$ &  $T_{[h]}$ & $\varphi_1(\sigma'^0(h))$ \\\hline\hline
        (1) & $e\in \HH_1(V)$ & $\sigma(e)$ & $T_V$ & $T_V$ & 
        $(\zeta_{s(e)} \rho_{C_{e,\sigma(e)}},\ 0)$ \\ \hline 
        (2) & $e\in \HH_2(V)$ & $\sigma(e)$ & $T_V$ & $T_V$ & 
        $(\rho_{C_{e,\sigma(e)}},\ \eta_{e,\sigma(e)})$ \\        
        (3) & $e\in \HH_3(V)$ & $\overline{\p(e)}$ & $P_{[\p(e)]}$ & $T_{V}$ & $(0,\ \iota_e)$ \\ \hline 
        (4)  & $f\in \HH_4(V)$ & $x_f$ & $T_V$ & $P_{[f]}$ & $(0,\ \rho_{C_{f,\n(f)}} \pi_{x_f})$ \\
        (5) & $f\in \HH_5(V)$ & $\n(f)$ & $P_{[\n(f)]}$ & $P_{[f]}$ & $(0,\ \rho_{C_{f,\n(f)}})$\\
    \end{tabular}\vspace{2mm}
    \caption{The definition of $\varphi'_1(\sigma'^0(h))$}
    \label{table:hom}
\end{table}

We show that all maps in Table \ref{table:hom} are homomorphisms in the homotopy category. 
The map of each case is given by the following diagram respectively. 
\begin{equation*}
    \begin{tabular}{cccccc}
    (1) & (2) & (3) & (4)& (5) \\ 
$\xymatrix@C=27pt@R=32pt{
P_V \ar@{->}[r]^{d_V} \ar@{->}[d]|{\zeta_{s(e)}\rho_{C_{e,\sigma(e)}}} 
& T^0_V \ar@{->}[d]|0 \\ 
P_V \ar@{->}[r]^{d_V} & T^0_V
}$     
&$\xymatrix@C=27pt@R=32pt{
P_V \ar@{->}[r]^{d_V} \ar@{->}[d]|{\rho_{C_{e,\sigma(e)}}} 
& T^0_V \ar@{->}[d]|{\eta_{e,\sigma(e)}} \\ 
P_V \ar@{->}[r]^{d_V} & T^0_V
}$ &
$\xymatrix@C=27pt@R=32pt{
0 \ar@{->}[r]^{0} \ar@{->}[d]|{0} 
&   P_{[\p(e)]}\ar@{->}[d]|{\iota_e} \\ 
P_V \ar@{->}[r]^{d_V} &  T^0_V
}$ &
$\xymatrix@C=27pt@R=32pt{
P_V \ar@{->}[r]^{d_V} \ar@{->}[d]|{0} 
& T_{V}^0 \ar@{->}[d]|{\rho_{C_{f,\n(f)}}\pi_{x_f}} \\ 
0 \ar@{->}[r]^{d_V} & P_{[f]}
}$ & 
$\xymatrix@C=27pt@R=32pt{
0 \ar@{->}[r]^{0} \ar@{->}[d]|{0} 
& P_{[\n(f)]} \ar@{->}[d]|{\rho_{C_{f,\n(f)}}} \\ 
0 \ar@{->}[r]^{0} & P_{[f]}
}$ 
\end{tabular}
\end{equation*}
The assertion is clear for (3) and (5). 
For the case (1), the assertion is immediate from (BC2). 
The commutativity for (2) and (4) follow from 
$$
d_V\circ \rho_{C_{e,\sigma(e)}}= \sum_{h\in \HH_2(V)\sqcup \HH_3(V)}\iota_h \rho_{C_{\p(h),h}}\rho_{C_{e,\sigma(e)}} 
\overset{\rm (BC2)}{=} 
\iota_e \rho_{C_{\p(e),\sigma(e)}}
=\eta_{e,\sigma(e)}\circ d_V 
$$ 
and 
$$
\rho_{C_{f,\n(f)}} \pi_{x_f} \circ d_V = 
\sum_{h\in \HH_2(V)\sqcup \HH_3(V)} \rho_{C_{f,\n(f)}} \pi_{x_f} \iota_h \rho_{C_{\p(h),h}} = \rho_{C_{f,\n(f)}} \rho_{C_{\p(x_f),x_f}} \overset{\rm (BC2)}{=} 0
$$
respectively.

As a consequence, a collection 
$\{\varphi_1(\sigma'^0(h)) \mid h\in H'\}$ of homomorphisms uniquely gives rise to a homomorphism $\varphi \colon KQ_{\Gamma'} \to \End(T)$ of $K$-algebras as desired.
\end{proof}

\subsubsection{The induced map $\bar{\varphi}$}\label{subsub:the induced map}

Next, we prove the following claim. 

\begin{proposition}\label{prop:induced by varphi} 
The map $\varphi$ in Proposition \ref{construction of varphi} factors through the natural surjection $\pi\colon KQ_{\Gamma'}\to A_{\Gamma'}$, giving $\overline{\varphi} \colon A_{\Gamma'}\to \End(T)$ such that $\varphi = \overline{\varphi} \circ \pi$.
\end{proposition}

To prove this, we need a few preparation. 
Recall from Definition \ref{def:BC} that, for two $h,f\in H$, if $f = \sigma'^m(h)$ for some $1\leq m \leq \val(s'(h))$, we have a path 
\begin{equation}\label{special cycle}
    C'_{h,f} \colon [h] \xrightarrow{\sigma'^0(h)} [\sigma'(h)] \xrightarrow{\sigma'(h)} [\sigma'^2(h)] \longrightarrow \cdots \longrightarrow 
    [\sigma'^{m-1}(h)] \xrightarrow{\sigma'^{m-1}(h)} [\sigma'^{m}(h)= f] 
\end{equation}
in the quiver $Q_{\Gamma'}$, where $\val(s'(h))$ is the valency of the vertex $s'(h)$ of $\Gamma'$. 
In particular, we have a cycle $C'_h:= C'_{h,h}$ in $Q_{\Gamma'}$. 
We compute the images of these cycles under the map $\varphi$.

\begin{lemma}\label{lem:image of cycles}
For each $h\in H$, the image of $C'_{h}$ under $\varphi$ is given by  
\begin{equation*}
    \varphi({C}'_{h}) = 
    \begin{cases}
        (\zeta_{s(h)}^{\val(s(h))} \rho_{C_{h}},\ 0) & \text{\rm ($h\in \HH_1(V)$),} \\ 
        (0,\ \iota_h \rho_{C_{\overline{\p(h)}}} \pi_{h}) & \text{\rm ($h\in \HH_2(V)\sqcup \HH_3(V)$),}\\ 
        (0,\ \rho_{C_{h}}) & \text{\rm ($h\in \HH_4(V)\sqcup \HH_5(V)$).} 
    \end{cases}
\end{equation*}
\end{lemma}
\begin{proof}
We first consider an angle $h\in \HH_1(V)$. In this case, we have 
$\sigma'^m(h)\in \HH_1(V)$ for all $1\leq m \leq \ell$, 
where $\ell := \val(s(h))=\val(s'(h))$ by Table \ref{table:flip}.  
Thus, we have 
\begin{equation*}
    \varphi(C_h') = \varphi(\sigma'^0(h))\cdots \varphi(\sigma'^{\val(s(h))-1}(h)) = (\zeta^{\val(s(h))}_{s(h)} \rho_{C_{h}},\ 0) 
\end{equation*}
as desired. 

Next, we show the claim for an angle $h\in \HH_2(V)\sqcup \HH_3(V)$. 
To see this, we uniquely decompose the cycle $C'_{h}$ in $Q_{\Gamma'}$ as
\begin{equation*}\label{cycle decomposition}
    C'_{h} = C'_{h,e_1'}
    \alpha_1 C'_{f_1,f_1'} \beta_1 C'_{e_2,e_2'} 
    \cdots 
    \alpha_i C'_{f_{i},f_{i}'} \beta_i C'_{e_{i+1},e_{i+1}'}
    \cdots 
    \alpha_t C'_{f_{t},f_{t}'} \beta_{t} C'_{e_{t+1},h} 
\end{equation*}
with the following property: 
\begin{itemize}
    \item $e_i,e_i'\in V$ and $f_i,f_i'\not \in V$, where we set $e_{i+rt} = e_{i}$ and $f_{i+rt} = f_{i}$ for $r\in \mathbb{Z}$.
    \item $\alpha_i := \sigma'^0(e_{i}') \colon e_{i}' \to f_i = \sigma'(e_i')$ is an arrow of $Q_{\Gamma'}$. In this case, we have $e_{i}'\in \HH_3(V)$.
    \item $\beta_i := \sigma'^{0}(f_{i}') \colon f_{i}' \to e_{i+1} = \sigma'(f_i')$ is an arrow of $Q_{\Gamma'}$. In this case, we have $f_{i}'\in \HH_4(V)$ with $x_{f_{i}'}=e_{i+1}$. 
    \item every angle appearing in $C_{e_i,e_{i}'}$ lies in $\HH_2(V)$ except for $e_{i}'$. 
    \item every angle appearing in $C_{f_{i},f_{i}'}$ lies in $\HH_5(V)$ except for $f_{i}'$. 
\end{itemize}

In this situation, Table \ref{table:hom} gives
\begin{eqnarray*}
    \varphi(\alpha_i) = (0,\, \iota_{e_i'}), \, 
    \varphi(\beta_i) = 
    (0,\, \rho_{C_{f_i',f_{i+1}}}\pi_{e_{i+1}}), \, 
    \varphi(C'_{e_{i},e_{i}'}) = 
    (\rho_{C_{e_{i},e_{i}'}},\, \eta_{e_{i},e_{i}'}) 
    \, \text{and}
    \,
    \varphi(C'_{f_i,f_i'}) = (0,\, \rho_{C_{f_i,f_i'}}). 
\end{eqnarray*}
Using this, we compute  
\begin{eqnarray*}
    \varphi(C'_{e_{i}',e_{i+1}'}) &=& 
    \varphi(\alpha_i)
    \varphi(C'_{f_{i},f_{i}'})
    \varphi(\beta_i)
    \varphi(C'_{e_{i+1},e_{i+1}'}) = (0,\, \iota_{e_{i}'} \rho_{C_{f_{i},f_{i+1}}} 
    \pi_{e_{i+1}'}) \quad \text{and} \\ 
    \varphi(C'_{f_{i}',f_{i+1}'}) &=& 
    \varphi(\beta_{i})
    \varphi(C'_{e_{i+1},e_{i+1}'})
    \varphi(\alpha_{i+1})
    \varphi(C'_{f_{i+1},f_{i+1}'})
    =(0,\, \rho_{C_{f_i',f_{i+1}'}}). 
\end{eqnarray*}
Consequently, we obtain
\begin{equation*}
    \varphi(C_h') = (0,\ \eta_{h,e_{1}'} \iota_{e_{1}'}
    \rho_{C_{f_1,f_{t+1}}} \pi_{e_{t+1}}\eta_{e_{t+1},h}) =
    (0,\ \iota_{h}\rho_{C_{f_1}} \pi_{h}) = 
    (0,\ \iota_h\rho_{C_{\overline{\p(h)}}}\pi_h) 
\end{equation*}
as desired, where we use the fact that $f_{1}= \overline{\p(h)}$ holds in $\Gamma$ by our construction. 

By a similar discussion as the previous paragraph, one can show the statement for $h\in \HH_4(V)\sqcup \HH_5(V)$. 
\end{proof}

\begin{proposition}\label{prop:image of socle}
    Keep the notations in Lemma \ref{lem:image of cycles}, we have  
\begin{equation*}
    \varphi(C_h'^{\m'(s'(h))}) = 
    \begin{cases}
        (-\rho_{C_{h}}^{\m(s(h))},\ 0) & \text{\rm ($h\in V$),} \\ 
        (0,\ \rho_{C_{h}}^{\m(s(h))}) & \text{\rm ($h\in H\setminus V$).} 
    \end{cases}
\end{equation*}
\end{proposition}

\begin{proof}
The multiplicity function $\m'$ is defined by \eqref{new m}. 
We compute the image by using Proposition \ref{lem:image of cycles}.
We recall that $V=\HH_1(V)\sqcup \HH_2(V) \sqcup \HH_3(V)$ and $H\setminus V = \HH_4(V) \sqcup \HH_5(V)$. 

If $h\in \HH_1(V)$, we get the assertion by 
\begin{equation*}
    \varphi({C}_{h}'^{\m'(s'(h))}) 
    \overset{\rm Lem\,\ref{lem:image of cycles}}{=} 
    (\zeta_{s(h)}^{\m(s(h))\val(s(h))}\rho_{C_{h}}^{\m(s(h))},\ 0) = 
    (-\rho_{C_{h}}^{\m(s(h))},\ 0)
\end{equation*}
where we use $\zeta_{s(h)}^{\m(s(h))\val(s(h))} = -1$ in the second equality. 

Next, if $h\in \HH_2(V)\sqcup \HH_3(V)$, we have 
\begin{equation*}
\varphi({C}_{h}'^{\m'(s'(h))}) 
\overset{\rm Lem.\,\ref{lem:image of cycles}}{=} (0,\ 
\iota_h \rho_{C_{\overline{\p(h)}}}^{\m(s(\overline{\p(h)})}
\pi_{h}) 
\overset{\rm (BC1)}{=} 
(0,\ \iota_h \rho_{C_{\p(h)}}^{\m(s(\p(h))}
\pi_{h}). 
\end{equation*}
Since $s(\p(h)) = s(h)$, it suffices to show that 
a homomorphism $(\rho_{C_h}, \iota_h \rho_{C_{\p(h)}} \pi_h)\in \Hom(T_V,T_V)$
is null-homotopic. 
In fact, there is a commutative diagram
\begin{equation*}
    \xymatrix@C=25pt@R=30pt{
        P_V \ar[rr]^{d_V} \ar[d]_{\rho_{C_h}} \ar@{}[rd]|{\circlearrowright} && T_V^0 \ar[d]^{\iota_h \rho_{C_{\p(h)}}\pi_h} \ar@{.>}[lld]_{\exists\,\kappa} \ar@{}[ld]|{\circlearrowright} \\ 
        P_V \ar[rr]^{d_V} && T_V^0
    }
\end{equation*}
by letting $\kappa := \rho_{C_{h,\p(h)}}\pi_h$. 
Thus, we get the desired equality in this case. 

Finally, the claims for $h\in \HH_4(V)\sqcup \HH_5(V)$ are obvious by Lemma \ref{lem:image of cycles}. 
\end{proof}
Now, we are ready to prove Proposition \ref{prop:induced by varphi}.

\begin{proof}[Proof of Proposition \ref{prop:induced by varphi}]
We show that $\varphi(I_{\Gamma'})=0$ for the two-sided ideal $I_{\Gamma'}$ in $KQ_{\Gamma'}$ defined by $\Gamma'$. 
We check each relation appearing in (BC1) and (BC2) for $\Gamma'$. 

\begin{itemize}
\item[(BC1)] 
Let $U$ be a polygon. 
For any $h,f\in U$, we have  
\begin{equation*}
    \varphi(C_h'^{\m'(s(h))} - C_f'^{\m'(s(f))}) = 
    \begin{cases}
    (-\rho_{C_{h}}^{\m(s(h))}+ \rho_{C_{f}}^{\m(s(f))},\ 0) & \text{if $U=V$,}\\ 
    (0,\ \rho_{C_h}^{\m(s(h))}-\rho_{C_f}^{\m(s(f))}) &\text{if $U\neq V$}
    \end{cases}
\end{equation*}
by Proposition \ref{prop:image of socle}. 
In all cases, it must be $0$ by (BC1) for $\Gamma$ as desired. 

\item[(BC2)] For (BC2), the assertion is obvious from the definition of $\varphi$ in Table \ref{table:hom}.
\end{itemize}
From the above argument, we get the claim. 
\end{proof}

\subsubsection{$\bar{\varphi}$ is an isomorphism}\label{subsub:isomorphism}
We show the next result, which immediately completes a proof of Theorem \ref{thm:main theorem}. 

\begin{theorem}\label{thm:varphi_isom}
    The map $\overline{\varphi}\colon A_{\Gamma'}\to \End(T)$ in Proposition \ref{prop:induced by varphi} is an isomorphism.
\end{theorem}

To prove this, we compare their dimension as $K$-vector spaces. 
Now, we decompose $T=T_V \oplus X$ and $A_{\Gamma'} = P_V' \oplus Y$.
That is, 
$X = \bigoplus_{U\neq V} P_U$ and  
$Y = \bigoplus_{U\neq V} P_U'$. 
We use the following description as matrix algebras. 

\begin{eqnarray*}
    \End(T) \cong 
    \begin{bmatrix}
    \End(T_V) & \Hom(X, T_V) \\ 
        \Hom(T_V, X) & 
        \End(X) 
    \end{bmatrix} 
    \quad \text{and} \quad 
    A_{\Gamma'}\cong \begin{bmatrix}
        \End(P_V') & \Hom(Y,P_V') \\ 
        \Hom(P_V', Y) & 
        \End(Y) 
    \end{bmatrix}. 
\end{eqnarray*}

\begin{proposition}\label{prop:dim_equals}
    The following statements hold. 
    \begin{enumerate}[\rm (1)]
        \item $(T_V, T_V) = (P_V',P_V').$ 
        \item $(T_V,X) = (P_V',Y).$
        \item $(X,T_V) = (Y,P_V').$
        \item $(X,X) = (Y,Y).$
    \end{enumerate}
    Consequently, $\dim_K\End(T) = \dim_K A_{\Gamma'}$. 
\end{proposition}

After proving Proposition \ref{prop:dim_equals}, we get Theorem \ref{thm:varphi_isom}. 

\begin{proof}[Proof of Theorem \ref{thm:varphi_isom}]
By definition, the socle of $A_{\Gamma'}$, regarded as a right $A_{\Gamma'}$-module, has a basis 
$\{C_{h}'^{\mathfrak{m}'(s'(h))} \mid h\in H'\}$ modulo $I_{\Gamma'}$. 
In Proposition \ref{prop:image of socle}, we have already seen that,
$\varphi(C_{h}'^{\mathfrak{m}'(s'(h))})\neq 0$ for all $h\in H'$. This shows that $\overline{\varphi}$ is injective. 
Then, it is automatically surjective by Proposition \ref{prop:dim_equals}. Thus, it is an isomorphism.  
\end{proof}

In the rest, we prove Proposition \ref{prop:dim_equals}. 
We describe $K$-dimension of the spaces of homomorphisms in terms of combinatorics on Brauer configurations. 

\begin{lemma}\label{lem:dim proj} 
    For two distinct polygons $U\neq W$ of $\Gamma$, we have 
    \begin{equation}\label{dim projUW}
        (P_U,P_W) = \displaystyle
        \sum_{v\in H/\langle\sigma\rangle} \mathfrak{m}(v) \occ(v,U)\occ(v,W). 
    \end{equation}
    On the other hand, we have 
    \begin{equation}\label{dim projV}
        (P_U,P_U) = 2-\#U + 
        \sum_{v\in H/\langle\sigma\rangle} \mathfrak{m}(v) \occ(v,U)^2. 
    \end{equation}
In particular, if $U$ is an edge, then $(P_U,P_U)$ is given by the same formula as \eqref{dim projUW} by letting $W=U$.
\end{lemma}

\begin{proof}
Let $U\neq W$ be polygons. 
By Proposition \ref{lem:Hom basis}, the set $\mathcal{C}(W,U)$ in \eqref{eq:C_WU} 
gives a $K$-basis of the space $\Hom(P_U,P_W)$. Thus, we have 
\begin{eqnarray*}
    (P_U,P_W) 
    = \#\mathcal{C}(W,U)  
    = \sum_{h\in W} \m(s(h))\occ(s(h),U) 
    = \sum_{v\in H/\langle\sigma\rangle} \m(v) \occ(v,U)\occ(v,W). 
\end{eqnarray*}
On the other hand, $\Hom(P_U,P_U)$ has a $K$-basis given by $\{1_U\} \cup \mathcal{C}(U,U)$, where $1_U$ denotes a constant path at $U$ in $Q_{\Gamma}$. 
Since $C_{h}^{\m(s(h))} = C_{f}^{\m(s(f))}$ holds for any $h,f\in U$ by (BC1), we obtain 
\begin{eqnarray*}
    (P_U,P_U) 
    &=& 1 + \#\mathcal{C}(U,U) - (\#U-1)  \\
    &=& 2- \#U + \sum_{h\in U} \m(s(h))\occ(s(h),U) \\ 
    &=& 2-\#U + \sum_{v\in H/\langle\sigma\rangle} \m(v) \occ(v,U)^2 
\end{eqnarray*}
as desired. 
\end{proof}

Next, let $\chi(U)$ be the multiplicity of $P_U$ as direct summands of $T_V^{0}$, that is,  
\begin{equation*}
    T_V^{0} =  \bigoplus_{e\in \HH_2(V)\sqcup \HH_3(V)}P_{[\p(e)]}  = \bigoplus_{U\in H/\psi} P_U^{\chi(U)}.  
\end{equation*}
Notice that, $\chi(U)\neq 0$ implies that $U$ is an edge by the condition (E).  

For each vertex $v$ of $\Gamma$, let 
\begin{equation*}
    \occ'(v,V) := \#\{e\in \HH_2(V)\sqcup \HH_3(V)\mid s(\overline{\p(e)}) = v\}. 
\end{equation*}
We remark that $\occ'(v,V)$ is precisely the number $\occ(v',V)=\#\{e\in V \mid s'(e) = v'\}$ considered in $\Gamma'$ for the polygon $V$ and the vertex $v'$ corresponding to $v$ (see Table \ref{table:flip}). 

\begin{lemma}\label{lem:occ}
For each vertex $v$ of $\Gamma$ such that $s^{-1}(v)\not\subset V$, 
we have 
    \begin{equation}\label{occ-occ}
        \sum_{U\in H/\psi} \chi(U) \occ(v,U) = \occ(v,V) + \occ'(v,V). 
    \end{equation}
\end{lemma}

\begin{proof}
For $h\in H$, we have 
\begin{equation*}
    \chi([h]) = \#\{e\in \HH_2(V)\sqcup \HH_3(V) \mid \p(e)=h\} + \#\{e\in \HH_2(V)\sqcup \HH_3(V)\mid \overline{\p(e)} = h\}. 
\end{equation*}
Using this, we have 
\begin{eqnarray*}
    \sum_{U\in H/\psi} \chi(U)\occ(v,U) 
    &=& 
    \sum_{\substack{h\in H \\ s(h)=v}} \chi([h]) \\ 
    &=& 
    \sum_{\substack{h\in H \\ s(h)=v}} 
    \#\{e\in\HH_2(V)\sqcup \HH_3(V)\mid \p(e)=h\} + \\ 
    && 
    \sum_{\substack{h\in H \\ s(h)=v}} 
    \#\{e\in \HH_2(V)\sqcup \HH_3(V)\mid \overline{\p(e)}=h\} \\
    &=& \occ(v,V) + \occ'(v,V)
\end{eqnarray*}
as desired. 
\end{proof}

Now, we are ready to prove Proposition \ref{prop:dim_equals}.

\begin{proof}[Proof of Proposition \ref{prop:dim_equals}]
We prove (1)-(4) in the statement. 

(1) Applying Lemma \ref{lem:dim proj}, we have 
\begin{equation}\label{dim:(1)-0}
    (P_V',P_V') = 2-\# V + \sum_{v\in H/\langle\sigma\rangle} \m(v)\occ'(v,V)^2.  
\end{equation}
On the other hand, since $T_V$ is a two-term pretilting complex, we have 
\begin{eqnarray}
    (T_V, T_V) = (T_V^0,T_V^0) - 2(P_V,T_V^0) + (P_V,P_V) 
\end{eqnarray}
by (\ref{dim formula}). 
We compute each summand. 
By the condition (E), every direct summand $P_U$ of $T_V^0$ is given by the edge $U$. 
Using Lemma \ref{lem:dim proj}, we have the following equality. 
Firstly, we have 
\begin{eqnarray}\nonumber
    (T_V^0,T_V^0) 
    &=& 
    \sum_{U,W\in H/\psi} \chi(U)\chi(W)(P_U,P_W) \\ \nonumber
    &=& 
    \sum_{U,W\in H/\psi} \chi(U)\chi(W)
    \left( 
    \sum_{v\in H/\langle\sigma\rangle} \m(v)\occ(v,U)\occ(v,W) 
    \right) \\ \nonumber
    &=& 
    \sum_{v\in H/\langle\sigma\rangle} \m(v)
    \left( 
    \sum_{U,W\in H/\psi} \chi(U)\chi(W) \occ(v,U)\occ(v,W) 
    \right) \\ \nonumber
    &=& 
    \sum_{v\in H/\langle\sigma\rangle} \m(v)
    \left( 
    \sum_{U\in H/\psi} \chi(U)\occ(v,U) 
    \right)^2 \\ \label{dim:(1)-1}
    &\overset{\eqref{occ-occ}}{=}& 
    \sum_{v\in H/\langle\sigma\rangle} \m(v)
    \left( 
    \occ(v,V) + \occ'(v,V) 
    \right)^2. 
\end{eqnarray}
Secondly, we have 
\begin{eqnarray}\nonumber
    (P_V, T_V^0) &=& \sum_{U\in H/\psi}\chi(U)(P_V,P_U) \\ 
    \nonumber
    &=& 
    \sum_{U\in H/\psi} \chi(U)
    \left( 
    \sum_{v\in H/\langle\sigma\rangle} \m(v)\occ(v,V) \occ(v,U) 
    \right) \\ \nonumber
    &=& 
    \sum_{v\in H/\langle\sigma\rangle} \m(v)\occ(v,V)
    \left( 
    \sum_{U\in H/\psi} \chi(U)\occ(v,U)  
    \right) \\ \nonumber
    &\overset{\eqref{occ-occ}}{=}& 
    \sum_{v\in H/\langle\sigma\rangle} \m(v)\occ(v,V)
    \left( 
    \occ(v,V) + \occ'(v,V)
    \right) \\ 
    &=& \label{dim:(1)-2} 
    \sum_{v\in H/\langle\sigma\rangle} \m(v)
    \left( 
    \occ(v,V)^2 + \occ(v,V) \occ'(v,V)
    \right). 
\end{eqnarray}
Finally, we have 
\begin{eqnarray} \label{dim:(1)-3}
    (P_V,P_V) 
    =   
    2- \#V + \sum_{v\in H/\langle\sigma\rangle} \m(v)\occ(v,V)^2.
\end{eqnarray}
Combining \eqref{dim:(1)-0}-\eqref{dim:(1)-3}, we get the desired equality (1). 

\noindent
(2) By the definition of flip, we have 
\begin{equation}\label{dim:(2)-0}
    (P_V', Y) = \sum_{v\in H/\langle\sigma\rangle} 
    \m(v) 
    \left( 
    \val(v) - \occ(v,V)
    \right) \occ'(v,V). 
\end{equation}
On the other hand, we have 
\begin{equation} 
    (T_V, X) = (T_V^0,X) - (P_V,X) 
\end{equation}
by (\ref{dim formula}). 
We have 
\begin{eqnarray}\nonumber
    (T_V^0,X) &=& \sum_{U\in H/\psi} \chi(U) (P_U,X) \\ \nonumber
    &=& 
    \sum_{U\in H/\psi} \chi(U) 
    \left(
    \sum_{v\in H/\langle\sigma\rangle} \m(v) \occ(v,U) 
    \left(
    \val(v) - \occ(v,V)  
    \right)
    \right)\\ \nonumber
    &=&
    \sum_{v\in H/\langle\sigma\rangle} \m(v) (\val(v)-\occ(v,V)) 
    \left(
    \sum_{U\in H/\psi} \chi(U)\occ(v,U) 
    \right)\\ \label{dim:(2)-2}
    &=& 
    \sum_{v\in H/\langle\sigma\rangle} \m(v) (\val(v)-\occ(v,V)) 
    \left(
    \occ(v,V) + \occ'(v,V) 
    \right) 
\end{eqnarray}
and  
\begin{equation}\label{dim:(2)-3}
    (P_V,X) = \sum_{v\in \in H/\langle\sigma\rangle} \m(v)(\val(v)-\occ(v,V))\occ(v,V).  
\end{equation}
By \eqref{dim:(2)-0}-\eqref{dim:(2)-3}, we obtain the desired equality (2) in the statement.

(3) Recall from Proposition \ref{prop:BCA is SMA} that Brauer configuration algebras $A_{\Gamma}$ and $A_{\Gamma'}$ are symmetric algebras. 
By \eqref{Serre duality}, we have 
$(P_V',Y) = (Y,P_V')$ and $(T_V,X) = (X,T_V)$. 
Then, (3) follows from (2).

(4) It is obvious from the definition of flip. In fact, Brauer configurations $\Gamma$ and $\Gamma'$ have the same structure except at $V$. 

We finish a proof of Proposition \ref{prop:dim_equals}.
\end{proof}

We end this paper with comments. As we show in Theorem \ref{thm:main theorem}, the condition (E) is a sufficient condition that a left mutation of a Brauer configuration algebra is again a Brauer configuration algebra, while this is not a necessary condition in general. In fact, we have the next example.

\begin{example}   
Let $A_{\Gamma}$ be a Brauer configuration algebra whose Brauer configuration is multiplicity-free and given by 
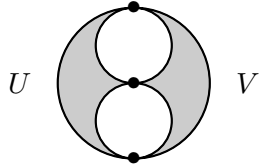
\begin{equation*}
\begin{tikzpicture}
    \coordinate (0) at(0,0);
    \coordinate (u) at(0,1);
    \coordinate (d) at(0,-1);
    
    \draw[fill = white!20!lightgray] (u)arc(90:270:0.5)(0)arc(90:270:0.5)(d)arc(270:90:1)(u);
    \draw[fill = white!20!lightgray] (u)arc(90:-90:0.5)(0)arc(90:-90:0.5)(d)arc(-90:90:1)(u);
    \draw[thick] (u)arc(90:270:0.5)(0)arc(90:270:0.5)(d)arc(270:90:1)(u);
    \draw[thick] (u)arc(90:-90:0.5)(0)arc(90:-90:0.5)(d)arc(-90:90:1)(u);
    
    \node at (0) {$\bullet$};
    \node at (u) {$\bullet$};
    \node at (d) {$\bullet$};
    \node at (-1.5,0) {$U$};
    \node at (1.5,0) {$V$};
\end{tikzpicture}  
\end{equation*}
That is, both $U$ and $V$ are $3$-gons. 
In this case, $A_{\Gamma}$ is isomorphic to the trivial extension of the path algebra of the $3$-Kronecker quiver $1\Rrightarrow 2$ with a minimal co-generator. By a direct calculation, we have $\End(\mu_{P_{X}}^-(A_{\Gamma}))\cong A_{\Gamma}$ as $K$-algebras for $X\in \{U,V\}$. 
In particular, it is again a Brauer configuration algebra, while $X$ does not satisfy the condition (E). 
\end{example}

On the other hand, one can use iterated flip to compute derived equivalent Brauer configuration algebras, as like (double-)star theorem for Brauer graph algebras \cite[Theorem 1.4]{Aihara15}. 
The next observation was given in \cite[Proposition 7.3]{Duffield18}. 

\begin{example}
    We consider the class of Brauer configuration algebras whose Brauer configurations have exactly one $3$-gon and no $m$-gons with $m>3$. 
    Applying iterated flip, we can deduce that 
    they are derived equivalent to one of Brauer configuration algebras whose Brauer configurations are of the form 
    
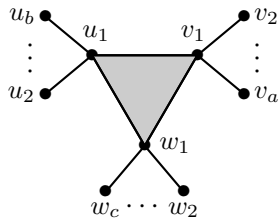
\begin{equation*}
        \begin{tikzpicture}
            \coordinate (a) at (30:0.8);
            \coordinate (b) at (30+120:0.8);
            \coordinate (c) at (30+240:0.8);
            \node at (a) {$\bullet$}; 
            \node at (b) {$\bullet$}; 
            \node at (c) {$\bullet$}; 
            \draw[thick, fill = white!20!lightgray] (a)--(b)--(c)--cycle; 
            \draw[thick] (a)--(b)--(c)--cycle;
            
            \coordinate (as) at ($(a)+ (40:0.8)$);
            \node at ($(a)+ (0:0.8)$) {\rotatebox{90}{$\cdots$}}; 
            \coordinate (at)at ($(a)+ (-40:0.8)$);
            \node at (as) {$\bullet$};
            \node at (at) {$\bullet$};
            \node at ($(a) + (100:0.3)$) {\small $v_1$};
            \node at (as) [right]{\small $v_2$};
            \node at (at) [right]{\small $v_a$};
            \draw[thick] (a)--(as) (a)--(at); 
            \coordinate (bs) at ($(b) + (40:-0.8)$);
            \node at ($(b)+ (0:-0.8)$) {\rotatebox{90}{$\cdots$}}; 
            \coordinate (bt)at ($(b) + (-40:-0.8)$);
            \node at (bs) {$\bullet$};
            \node at (bt) {$\bullet$};
            \node at ($(b) + (80:0.3)$) {\small $u_1$};
            \node at (bs) [left]{\small $u_2$};
            \node at (bt) [left]{\small $u_b$};
            \draw[thick] (b)--(bs) (b)--(bt); 
            \coordinate (cs) at ($(c)+ (40+90:-0.8)$);
            \node at ($(c)+ (0+90:-0.8)$) {\rotatebox{0}{$\cdots$}}; 
            \coordinate (ct)at ($(c)+ (-40+90:-0.8)$);
            \node at (cs) {$\bullet$};
            \node at (ct) {$\bullet$};
            \node at ($(c) + (0:0.4)$) {\small $w_1$};
            \node at (cs) [below]{\small $w_2$};
            \node at (ct) [below]{\small $w_c$};
            \draw[thick] (c)--(cs) (c)--(ct); 
            
        \end{tikzpicture}
    \end{equation*}
\end{example}

\section*{Acknowledgements} 
T.A is supported by JSPS Grant-in-Aid for Transformative Research Areas (A) 22H05105. 
Y.Z is supported by the NSFC (Grant No. 12201211) and the China Scholarship Council (Grant No. 202109710002). The authors would like to thank Professor Osamu Iyama for his careful guidance and encouraging discussions which leads to this collaboration. Y.Z thanks the support and excellent working conditions during her visit to Professor Osamu Iyama at the University of Tokyo from November 2022 to November 2023.

\bibliographystyle{abbrv} 

\end{document}